\documentclass[12pt]{amsart}

\usepackage[english]{babel}
\usepackage[pdftex,textwidth=410pt,marginratio=1:1]{geometry}
\usepackage[dvips]{graphics}
\usepackage[colorinlistoftodos]{todonotes}
\usepackage[curve]{xypic}
\usepackage{amsfonts,amsmath,amsthm,amssymb,bbm,cancel,color,graphicx}

\newtheorem{theorem}{Theorem}[section]
\newtheorem{corollary}[theorem]{Corollary}
\newtheorem{proposition}[theorem]{Proposition}
\newtheorem{lemma}[theorem]{Lemma}
\newtheorem{conjecture}[theorem]{Conjecture}

\theoremstyle{definition}

\newtheorem{example}[theorem]{Example}
\newtheorem{remark}[theorem]{Remark}

\theoremstyle{property}

\def\={\;=\;}
\def\+{\,+\,}

\DeclareFontFamily{OT1}{rsfs}{}
\DeclareFontShape{OT1}{rsfs}{n}{it}{<-> rsfs10}{}
\DeclareMathAlphabet{\curly}{OT1}{rsfs}{n}{it}

\renewcommand\k{\mathsf k}
\newcommand\Hk{H_{\mathsf k}}

\newcommand\sfK{\mathsf K}

\newcommand\LL{\mathbb L}

\renewcommand\O{\mathcal O}
\newcommand\PP{\mathbb P}

\newcommand\cP{\mathcal P}
\newcommand\cW{\mathcal W}

\newcommand\mdot{{\scriptscriptstyle\bullet}}

\newcommand\n{\mathbf{n}}

\newcommand\pt{\mathrm{pt}}
\newcommand\vir{\mathrm{vir}}
\newcommand\ver{\mathrm{ver}}

\newcommand\AJ{\mathrm{AJ}}

\newcommand\td{\mathrm{td}}

\newcommand\lam{\pmb{\lambda}}

\newcommand\Xb{\,\overline{\!X}}

\newcommand\C{\mathbb C}

\newcommand\cC{\mathcal C}

\newcommand\FF{\mathbb F}

\newcommand\II{\mathbb I}

\newcommand\sfZ{\mathsf Z}

\newcommand\Q{\mathbb Q}
\newcommand\cQ{\mathcal Q}

\newcommand\Z{\mathbb Z}

\newcommand\cZ{\mathcal Z}

\renewcommand\t{\mathfrak t}

\newcommand{\rt}[1]{\stackrel{#1\,}{\rightarrow}}
\newcommand{\Rt}[1]{\stackrel{#1\,}{\longrightarrow}}
\newcommand\To{\longrightarrow}
\newcommand\into{\hookrightarrow}
\newcommand\Into{\ensuremath{\lhook\joinrel\relbar\joinrel\rightarrow}}
\newcommand\INTO{\ar@{^{(}->}[r]}
\newcommand\Mapsto{\ensuremath{\shortmid\joinrel\relbar\joinrel\rightarrow}}

\newcommand\ip{\,\lrcorner}

\renewcommand\_{^{}_}

\newcommand\bull{{\scriptscriptstyle\bullet}}
\newcommand\udot{^\bull}

\newcommand\rk{\operatorname{rk}}

\newcommand\tr{\operatorname{tr}}

\newcommand\ch{\operatorname{ch}}
\newcommand\ev{\operatorname{ev}}

\newcommand\id{\operatorname{id}}

\newcommand\Hom{\operatorname{Hom}}
\renewcommand\hom{\curly H\!om}
\newcommand\Ext{\operatorname{Ext}}
\newcommand\ext{\curly Ext}
\newcommand\At{\operatorname{At}}
\newcommand\Aut{\operatorname{Aut}}

\newcommand\Ob{\operatorname{Ob}}
\newcommand\Bl{\operatorname{Bl}}
\newcommand\Pic{\operatorname{Pic}}

\newcommand\Spec{\operatorname{Spec}\,}
\newcommand\Hilb{\operatorname{Hilb}}

\newcommand\Sym{\operatorname{Sym}}
\newcommand\beq[1]{\begin{equation}\label{#1}}
\newcommand\eeq{\end{equation}}
\newcommand\beqa{\begin{eqnarray*}}
\newcommand\eeqa{\end{eqnarray*}}

\DeclareRobustCommand{\SkipTocEntry}[4]{}
\makeatletter
\def\l@section{\@tocline{1}{0pt}{1pc}{}{}}
\def\l@subsection{\@tocline{2}{0pt}{1pc}{4.6em}{}}
\def\l@subsubsection{\@tocline{3}{0pt}{1pc}{7.6em}{}}
\renewcommand{\tocsection}[3]{%
  \indentlabel{\@ifnotempty{#2}{\makebox[2.3em][l]{%
    \ignorespaces#1 #2.\hfill}}}#3}
\renewcommand{\tocsubsection}[3]{%
  \indentlabel{\@ifnotempty{#2}{\hspace*{2.3em}\makebox[2.3em][l]{%
    \ignorespaces#1 #2.\hfill}}}#3}
\renewcommand{\tocsubsubsection}[3]{%
  \indentlabel{\@ifnotempty{#2}{\hspace*{4.6em}\makebox[3em][l]{%
    \ignorespaces#1 #2.\hfill}}}#3}
\makeatother

\begin{document}
\title[Stable pairs on local surfaces I: vertical]{\ \\ \vspace{-17mm} Stable pairs with descendents on local surfaces I: the vertical component\vspace{-2mm}}
\author[M.~Kool and R.~P.~Thomas]{Martijn Kool and Richard P.~Thomas \vspace{-8mm}}
 
\maketitle

\centerline{\emph{with an Appendix by Aaron Pixton and Don Zagier}}

\begin{abstract}
We study the full stable pair theory --- with descendents --- of the Calabi-Yau 3-fold $X=K_S$, where $S$ is a surface with a smooth canonical divisor $C$.

By both $\C^*$-localisation and cosection localisation we reduce to stable pairs supported on thickenings of $C$ indexed by partitions. We show that only strict partitions contribute, and give a complete calculation for length-1 partitions. The result is a surprisingly simple closed product formula for these ``vertical" thickenings.

This gives all contributions for the curve classes $[C]$ and $2[C]$ (and those which are not an integer multiple of the canonical class). Here the result verifies, via the descendent-MNOP correspondence, a conjecture of Maulik-Pandharipande, as well as various results about the Gromov-Witten theory of $S$ and spin Hurwitz numbers.
%
\end{abstract}
\thispagestyle{empty}

\renewcommand\contentsname{\vspace{-9mm}}
\tableofcontents

\section{Introduction} \label{intro}
Let $S$ be a smooth complex projective surface, and let $X = \mathrm{Tot}(K_S)$ be the total space of its canonical bundle $K_S$ with its natural action of $T=\C^*$ on the fibres. We use the natural maps
\begin{displaymath}
\xymatrix
{
X \ar_{\pi}[r] & S. \ar@/_/_{\iota}[l]
}
\end{displaymath}
For $\beta \in H_2(S,\Z)$ and $\chi\in\Z$ we let $P_X:=P_\chi(X,\iota_* \beta)$ denote the moduli space of stable pairs $(F,s)$ on $X$ \cite{PT1} with curve class $[F]=\iota_*\beta$ and holomorphic Euler characteristic $\chi(F)=\chi$.

The moduli space $P_X$ has a symmetric perfect obstruction theory \cite{PT1}, but is noncompact. The $T$-action induces one on $P_X$ with compact fixed point locus $P_{X}^{T}$. Therefore we can define the stable pair invariants of $X$ via $T$-equivariant virtual localisation \cite{GP}.\footnote{We emphasise that in this paper we are concerned with the full stable pair and Gromov-Witten invariants of $X$, \emph{not} their reduced cousins computed in \cite{KT1,KT2}.} See Section \ref{sectionfixedlocus} for a review of the details, and for the construction of the descendent insertions
$$
\tau_\alpha(\sigma) := \pi_{P*} \big( \pi_X^{*} \sigma \cap \ch_{\alpha+2}^{T}(\FF) \big) \in H^{*}_{T}(P_X,\Q)
$$
for $\alpha\geq0$. Here we use $\sigma$ to denote both a class in $H^*(S,\Q)$ and the corresponding class $\sigma\otimes1\in H_{T}^{*}(X,\Q) \cong H^*(S,\Q)\otimes\Q[t]$, where $t$ is the equivariant parameter. The resulting descendent invariants of $X$ live in $\Q[t,t^{-1}]$ and are defined by
\begin{equation} \label{generalinv}
P_{\chi,\beta}\big(X,\tau_{\alpha_1}(\sigma_1) \cdots \tau_{\alpha_m}(\sigma_m)\big) := \int_{[P_{\chi}(X,\beta)^{T}]^{\vir}} \frac{1}{e(N^{\vir})} \prod_{j=1}^{m} \tau_{\alpha_j}(\sigma_j)\big|\_{P_{\chi}(X,\beta)^{T}}\,.
\end{equation}
Many of these invariants vanish:

\begin{theorem} \label{main0}
If $S$ has a reduced, irreducible canonical divisor then
$$
P_{\chi,\beta}\big(X,\tau_{\alpha_1}(\sigma_1) \cdots \tau_{\alpha_m}(\sigma_m)\big)\ =\ 0
$$
unless $\beta$ is an integer multiple of the canonical class $\k$ and all $\sigma_i$ lie in $H^{\le2}(S)$.
\end{theorem}

More generally one can localise the calculation of $P_{\chi,\beta}\big(X,\tau_{\alpha_1}(\sigma_1) \cdots \tau_{\alpha_m}(\sigma_m)\big)$ to (thickenings of) a canonical divisor $C$. In the context of Seiberg-Witten and Gromov-Witten theory on $S$ this goes back to ideas of Witten, Taubes and Lee-Parker \cite{LP}, formalised in algebraic geometry as Kiem-Li's cosection localisation \cite{KL1,KL2,KL3}.

So from now on we consider only $S$ with \emph{a smooth connected canonical divisor}\footnote{In fact all we require, by the deformation invariance of stable pair and Gromov-Witten invariants, is that some deformation of $S$ should have such a divisor.} $C$. Because of Theorem \ref{main0} we need only work with curve classes $\beta=d\k,\ d\in\Z_{>0}$, which are integer multiples of the canonical class.
We use cosection localisation to further localise the $T$-fixed moduli space $P_X^T$ to thickenings of $C$ indexed by partitions $\lam=(\lambda_0,\lambda_1,\ldots,\lambda_{l-1})$ with $\lambda_0\ge\ldots\ge\lambda_{l-1}>0$ and $|\lam|=\sum\lambda_i=d$. The components\footnote{By convention a component means a union of connected components.} $P_{\lam C}^T$ of the localised moduli space parameterise stable pairs with support $\lam C$ defined by the ideal sheaf
$$
\O(-\lambda_0S)+I_C(-\lambda_1S)+I_C^2(-\lambda_2S)+\ldots+I_C^{l-1}(-\lambda_{l-1}S)+I_C^l.
$$
Here $\O_X(-S)$ is the ideal of the zero section of $K_S$, and $I_C=\pi^*\O(-C)$ is the ideal sheaf of $\pi^*C$.

\begin{example} The partition $\lam=(4,2,1)$ corresponds to the thickening $\lam C$ which, transverse to $C$, looks like \smallskip
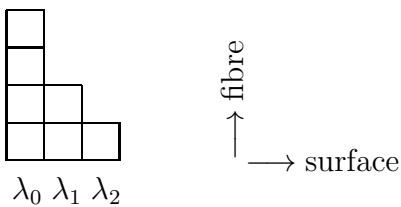
\begin{displaymath}
\xy
(0,0)*{}; (0,20)*{} **\dir{-} ; (5,0)*{} ; (5,20)*{} **\dir{-} ; (10,0)*{} ; (10,10)*{} **\dir{-} ; (15,0)*{} ; (15,5)*{} **\dir{-} ;  (0,0)*{} ; (15,0)*{} **\dir{-} ; (0,5)*{} ; (15,5)*{} **\dir{-} ; (0,10)*{} ; (10,10)*{} **\dir{-} ; (0,15)*{} ; (5,15)*{} **\dir{-} ; (0,20)*{} ; (5,20)*{} **\dir{-}; (8,-4)*{\lambda_0\;\lambda_1\;\lambda_2}; ; (42,0)*{\longrightarrow \mathrm{surface}} ; (30,8)*{\rotatebox{90}{$\longrightarrow \mathrm{fibre}$}}
\endxy
\end{displaymath}
\end{example}

\noindent In fact only \emph{strict} partitions ($\lambda_0>\ldots>\lambda_l$) like this one contribute.

\begin{theorem} \label{stric}
The integrals \eqref{generalinv} can be localised to integrals over the moduli spaces $P_{\lam C}^T\subset P_{\chi}(X,d\k)$ with $\lam\vdash d$ a {\bf strict} partition of $d=|\lam|$.
\end{theorem}

We form the generating function
\begin{equation} \label{ZPfull}
\sfZ_{d\k}^P(X,\tau_{\alpha_1}(\sigma_1) \cdots \tau_{\alpha_m}(\sigma_m))\,:=\,
\sum_{\chi\in\Z}P_{\chi,d\k}\big(X,\tau_{\alpha_1}(\sigma_1) \cdots \tau_{\alpha_m}(\sigma_m)\big) q^\chi
\end{equation}
in $\Q[t,t^{-1}](\!(q)\!)$, and let 
\begin{equation} \label{ZPver}
\sfZ_{d\k}^P(X,\tau_{\alpha_1}(\sigma_1) \cdots \tau_{\alpha_m}(\sigma_m))_{\ver}
\end{equation}
denote the contribution from length-1 partitions $\lam=(d)$ in Theorem \ref{stric}. This is a generating series of integrals over moduli spaces of stable pairs whose support has ideal $I_C(-dS)$. In particular they are contained in $\pi^{-1}(C)$ and we call them ``vertical". The main result of this paper is an algorithm for the computation of \eqref{ZPver} (Remark \ref{generalinsertions}) and a closed formula when all the insertions $\sigma_i$ are $H^2$ classes. We let $D_i$ denote their Poincar\'e dual classes (these are \emph{any} $H_2(S,\Q)$ classes, not necessarily divisors).

\begin{theorem} \label{main}
Suppose that $S$ has a smooth irreducible canonical divisor of genus $h=\k^2+1$. Set $Q:=-q$ and $|\alpha|:=\alpha_1+\ldots+\alpha_m$. Without descendents, $\sfZ_{d\k}^P(X)_{\ver}$ equals
$$
(-1)^{\chi(\O_S)} \Bigg(\!\frac{(-1)^d}{d^{d-1}}\!\Bigg)^{\!\!h-1}\!\!
\big(Q^{\frac d2}-Q^{-\frac d2}\big)^{2h-2}\
\prod_{i=1}^{d-1}\Big(\!(d-i)Q^{\frac d2}-dQ^{\frac d2 - i}+iQ^{-\frac d2}\Big)^{\!h-1}. $$
Adding descendents, $\sfZ_{d\k}^P(X,\tau_{\alpha_1}(D_1) \cdots \tau_{\alpha_m}(D_m))_{\ver}$ equals
$$
\sfZ_{d\k}^P(X)_{\ver\,}(dt)^{|\alpha|} \prod_{j=1}^{m}\frac{(d\k \cdot\! D_j)}{(\alpha_j +1)!}\ \frac{Q^{\frac d2(\alpha_j+1)} - Q^{-\frac d2(\alpha_j+1)}}{\big(Q^{\frac d2} - Q^{-\frac d2}\big)^{\alpha_j+1}}\,.
$$
\end{theorem}

\begin{remark}
We deduce that $\sfZ_{d\k}^P(X,\tau_{\alpha_1}(D_1) \cdots \tau_{\alpha_m}(D_m))_{\ver}$ is invariant under $q \leftrightarrow q^{-1}$ up to a factor $(-1)^{|\alpha|}$. In particular we get invariance under $q \leftrightarrow q^{-1}$ for primary insertions. In the cases $d=1,2$ these expressions calculate the full generating function \eqref{ZPfull}. 
\end{remark}


The MNOP correspondence \cite{MNOP,PT1} conjectures that the Gromov-Witten and stable pairs theories of $X$ determine one another.\footnote{While the original MNOP correspondence dealt with Gromov-Witten theory and DT theory, in this paper we always mean its simpler reformulation in \cite{PT1} in the language of stable pairs. This is critical when descendents are included, as the form of the conjecture for DT theory is still unknown.} This has been upgraded by Pandharipande-Pixton \cite{PP1, PP2} to a correspondence of full descendent theories. This descendent-MNOP conjecture is more complicated than the original MNOP conjecture, involving a certain inexplicit matrix $\widetilde\sfK_{\mu\nu}$. Pandharipande-Pixton have proved their conjecture in many cases, but not for the local general type surfaces of this paper. So in Sections \ref{GWX}, \ref{GWS} we \emph{assume} the descendent-MNOP correspondence and apply it to our results. Firstly this gives (see Theorem \ref{GWmain0}) the obvious vanishing result analogous to Theorem \ref{main0} for the Gromov-Witten generating function 
\begin{equation} \label{ZGWfull}
\sfZ_{\beta}^{GW}(X,\tau_{\alpha_1}(\sigma_1) \cdots \tau_{\alpha_m}(\sigma_m))\ \in\ \Q[t,t^{-1}](\!(u)\!).
\end{equation}
Next we consider the vertical contribution of Theorem \ref{main} to the stable pairs generating function for $\beta=d\k$. Pushing it through the descendent-MNOP conjecture we get a contribution to the Gromov-Witten theory which we call
\begin{equation} \label{ZGWver}
\sfZ_{d\k}^{GW}(X,\tau_{\alpha_1}(D_1) \cdots \tau_{\alpha_m}(D_m))_{\ver\,},
\end{equation} 
which is the full generating function for $d=1,2$. Its computation has several applications:
\begin{itemize}
\item In Theorem \ref{GWmain} we prove that \eqref{ZGWver} is the product of the generating function without insertions $\sfZ_{d\k}^{GW}(X)_{\ver}$ and a formal Laurent series in $u$ depending only on $d,\ d\k\!\cdot\!D_j$ and the descendence degrees $\alpha_1, \ldots, \alpha_m$. 
\item The lowest order term in $u$ of \eqref{ZGWfull} has coefficient the descendent Gromov-Witten invariant
\begin{equation} \label{surfin}
N\udot_{g,\beta}(S,\tau_{\alpha_1}(\sigma_1)\cdots\tau_{\alpha_m}(\sigma_m))
\end{equation}
of the surface $S$ in genus
\beq{jeenuss}
g\,:=\ 1 - \int_{\beta} c_1(S) - m + \sum_{j=1}^{m} \Big(\alpha_j + \frac12\deg(\sigma_j)\Big).
\eeq
Here $\sigma_j \in H^{\deg \sigma_j}(S,\Q)$, and the invariant is zero if \eqref{jeenuss} is not an integer.
This invariant \eqref{surfin} satisfies the same vanishing as its 3-fold analog (Corollary \ref{GWSvanish}). In the case of no insertions Lee-Parker \cite{LP} proved that \eqref{surfin} is equal to the degree $d$ unramified spin Hurwitz number of $C$ with theta characteristic $K_S|_C$. This result was proved algebro-geometrically by Kiem-Li \cite{KL1, KL2}.

The spin Hurwitz numbers were recently computed explicitly using TQFT by Gunningham \cite{Gun}. Our vertical contribution correctly reproduces the part of his formula corresponding to length-1 partitions (Corollary \ref{Guncor}).  Again this is the whole thing when $d=1,2$.  It is mysterious how the MNOP conjecture matches up the very different occurrences of these partitions in the two theories.
\item The descendent-MNOP correspondence involves a universal matrix $$\widetilde\sfK_{\mu\nu} \in \Q[i,c_1,c_2,c_3](\!(u)\!),$$ where $\mu, \nu$ run over all partitions and $i^2=-1$. Proposition \ref{PaPix} shows that for local surfaces with irreducible reduced canonical divisor and $\deg\sigma_i\ge2$ we only need to know the specialisation
\beq{simplf}
\widetilde\sfK_{\mu\nu}\Big|_{c_1 = t,\,c_2=c_3=0}\quad\mathrm{for}\ \mu,\nu\ \mathrm{of\ length\ one.}
\eeq
In this case writing $\mu=(a),\,\nu=(b)$,  the specialisation \eqref{simplf} equals
$$
t^{a-b} \cdot f_{ab}(u)\quad\mathrm{for\ some\ }f_{ab}(u) \in \Q[i](\!(u)\!)
$$
by \cite{PP1}. We conjecture that $f_{ab}(u)$ is a Laurent \emph{monomial} of degree $1-a$ (Conjecture \ref{conj1}). Assuming this we show the $f_{ab}(u)$ are uniquely determined by the fact that the Gromov-Witten generating function \eqref{ZGWfull} starts in the correct degree. They then uniquely determine the surface invariants \eqref{surfin}, confirming (Corollary \ref{MPcor}) old conjectural formulae of Maulik-Pandharipande \cite{MP}. Maulik-Pandharipande's formulae were first proved on the Gromov-Witten side by Kiem-Li \cite{KL1, KL2} and later J.~Lee in symplectic geometry \cite{Lee}. Our calculation via descendent-MNOP requires a combinatorial identity we found experimentally in Maple and Mathematica, and which is proved in Appendix by A.~Pixton and D.~Zagier (Theorem \ref{appthm}). 
\end{itemize}

\begin{remark}
This paper only considers the vertical component of the zero locus of the cosection in $P_{X}^{T}$. In a sequel \cite{KT4} we calculate the contribution of the other components in the case of \emph{bare curves} (i.e.~minimal $\chi$, so that the stable pairs have no cokernel or ``free points"). This turns out to explain part of the structure of S.~Gunningham's formula \cite{Gun} from the stable pairs point of view. 
\end{remark}

\noindent \textbf{Relations to older work.} The results of this paper can be seen as being precisely \emph{orthogonal} to the earlier work \cite{KT1, KT2} on reduced classes. There we also considered stable pair invariants on $X = K_S$ for $S$ a surface with holomorphic 2-forms: $h^{2,0}(S)>0$. But we worked only with effective curve classes $\beta$ for which the Noether-Lefschetz locus has the \emph{expected codimension} $h^{2,0}(S)$. The standard invariants (Gromov-Witten, stable pairs) therefore vanish, and we get interesting \emph{reduced} invariants only by reducing the obstruction bundle in a canonical way.

Here we study the standard (nonreduced) GW/stable pairs invariants. These need not vanish for curve classes whose Noether-Lefschetz locus has the ``wrong" codimension. We find the only classes which contribute are multiples of the canonical class $\k$ (whose Noether-Lefschetz locus has codimension 0).

The papers \cite{KT1,KT2} also focused on the \emph{horizontal component} of the moduli space of stable pairs. This is the only component relevant for (sufficiently ample) enumerative problems on $S$ such as G\"ottsche's conjecture. Here the horizontal component does not contribute to the invariants and we study the \emph{vertical component} instead. There we derived universality results but no closed formula. Here we obtain a closed product formula when all insertions come from $H^2$ classes. \\

\noindent\textbf{Plan.} We localise $[P_X]^\vir$ first to its $T$-fixed locus, in Section \ref{sectionfixedlocus}, then further to pairs supported on thickenings of a canonical divisor $C$ in Section \ref{cosecsec}. This will be enough to prove Theorem \ref{main0}. The moduli space of $T$-fixed pairs supported on a \emph{vertically thickened} smooth curve $C$ is identified with a nested Hilbert scheme of $C$ in Section \ref{sectionneshilb}. Section \ref{virtsec} expresses the virtual cycle as a cycle on this nested Hilbert scheme. This is further simplified to an expression on a single symmetric product $\Sym^{n_0\!}C$ in Section \ref{vircycleII}.
In Sections \ref{virnorsec} and \ref{descendsec} we see how the virtual normal bundle and descendent integrands simplify on $\Sym^{n_0\!}C$. This allows us to compute the integrals in Sections \ref{step2} and \ref{step3} and derive Theorem \ref{main}. The formulae are rather lengthy and complicated at each stage, until right at the end they are summed up into a mysteriously simple closed product formula. This suggests one should work with the generating series, rather than individual invariants, from the beginning, but we have not found a way to do this. Finally Sections \ref{GWX} and \ref{GWS} discuss applications to the Gromov-Witten invariants of $X$ and $S$ respectively. \medskip

\noindent\textbf{Notation.} Given any map $f\colon A\to B$ we also use $f$ for the induced map $f\times\id_C\colon A\times C\to B\times C$.
We suppress various pullback maps for clarity of exposition. We denote the cohomology class Poincar\'e dual to a cycle $A$ by $[A]$. We use ${}^\vee$ for derived dual of complexes, and ${}^*$ for the underived dual $\hom(\ \cdot\ ,\O)$ of coherent sheaves.

We use the standard conventions for (possibly negative) binomial coefficients. That is
\beq{binomconv}
{n\choose k}\ \ \mathrm{is\ defined\ to\ be\ }\ (-1)^k{k-n-1\choose k}\
\mathrm{\ when\ }\ n<0,k\ge0,
\eeq
and it is defined to be zero whenever $k<0$ or $k>n\ge0$. The binomial theorem $(1+x)^n=\sum_{k\ge0}{n\choose k}x^k$ then holds for any $n\in\Z$.
\medskip

\noindent\textbf{Acknowledgements.} We thank Jim Bryan and Rahul Pandharipande for useful conversations. We are grateful to Aaron Pixton and Don Zagier for proving our conjecture (now Theorem \ref{appthm}) in Appendix A. We also thank Frits Beukers and Wadim Zudilin for discussions on this conjecture; in fact they found an independent proof after the Appendix was written. Finally, we warmly thank the referee for a careful reading of the manuscript and pointing out some mistakes.

Both authors were supported by EPSRC programme
grant EP/G06170X/1. Part of this research was done while the first author was a PIMS postdoc (CRG Geometry and Physics) at University of British Columbia, and on NWO-GQT funding and a Marie Sk{\l}odowska-Curie IF (656898) at Utrecht University. 

\section{$T$-localised stable pair theory} \label{sectionfixedlocus}

Let $P_X:=P_\chi(X,\iota_*\beta)$ denote the moduli space of stable pairs $(F,s)$ on $X$. It is a quasi-projective scheme whose product with $X$,
\begin{displaymath}
\xymatrix@=15pt
{
& P_X \times X \ar[dl]_{\pi_P} \ar[dr]^{\pi_X} & \\
P_X & & X,\!
}
\end{displaymath}
carries a universal sheaf $\FF$, section $s$ and universal complex
\[
\II\udot = \{\O\To\FF \}.
\]

The action of $T$ on $X$ induces one on $P_X$ with respect to which $\FF$ and $\II\udot$ are $T$-equivariant. Since the $T$-fixed locus
$$
P^T_X\subset P_X
$$
is compact we may use virtual localisation \cite{GP} to define stable pair invariants of $X$ via residue integrals over the virtual cycle of $P_X^T$.

To describe the virtual cycle, we view stable pairs $(F,s)$ as objects $I\udot:=\{\O_X\rt{s}F\}$ of $D(X)$ of trivial determinant as in \cite{PT1}. Then $P_X$ acquires a $T$-equivariant perfect symmetric obstruction theory \cite[Theorem 4.1]{HT}
\beq{pofs}
E\udot:=R\hom_{\pi_{P}}(\II\udot, \II\udot)_{0}^\vee[-1]\To\LL_{P_X}
\eeq
with obstruction sheaf
$$
\Ob_X\!:=\ext_{\pi_P}^{2}(\II\udot, \II\udot)\_0.
$$
Here $(\,\cdot\,)_0$ denotes trace-free part. By \cite{GP} the $T$-fixed locus $P_X^T$ inherits a perfect obstruction theory
\beq{Tdef}
\big(\!R\hom_{\pi_{P}}(\II\udot, \II\udot)_{0}^f\big)\!^\vee[-1]\To\LL_{P_X^T}
\eeq
with obstruction sheaf
$$
\Big(\!\Ob_X\!\big|\_{P^T_X}\Big)^f.
$$
Here $(\,\cdot\,)^f$ denotes the $T$-fixed part: the weight-0 part of the complex. 

The obstruction theory \eqref{Tdef} defines a virtual cycle on $P^T_X$ by \cite{BF,LT}. The $T$-localised invariants of $X$ are defined by integrating insertions against the cap product of $e(N^{\vir})^{-1}$ with this virtual cycle. Here the virtual normal bundle $N^{\vir}=\{V_0\to V_1\}$ is defined to be the part of \eqref{pofs} (dualised and restricted to $P^T_X$) with nonzero weights, and
\[
e(N^{\vir}):= \frac{c_{\mathrm{top}}^{T}(V_0)}{c_{\mathrm{top}}^{T}(V_1)} \in H_{T}^{*}(P_{X}^{T},\Q)\otimes_{\Q[t]}\Q[t,t^{-1}]\cong
H^*(P_{X}^{T},\Q) \otimes_{\Q} \Q[t,t^{-1}]
\]
is its $T$-equivariant virtual Euler class.\footnote{We may choose the $V_i$ to be $T$-equivariant vector bundles with no weight-$0$ parts, so that the $c_{\mathrm{top}}^{T}(V_i)$ are invertible in $H_{T}^{*}(P_{X}^{T},\Q)\otimes_{\Q[t]}\Q[t,t^{-1}]$.} As usual $$t:=c_1(\t)\in H^*(BT,\Q)\cong\Q[t]$$ denotes the first Chern class of the standard weight-1 representation $\t$ of $T$, the generator of the equivariant cohomology of $BT$.

In this paper we are interested in descendent insertions. The sheaf $\FF$ is $T$-equivariant, so we can consider its $T$-equivariant Chern classes
\[
\ch_{i}^{T}(\FF) \in H^{*}_{T}(P_X,\Q).
\]
Given any $\sigma\in H^*(S,\Q)$, we consider it as lying in $H^*_T(X,\Q)$ (or its localization at $t$) by identifying it with the element
\begin{equation} \label{HtoHT}
\sigma\otimes1\in H^*(S,\Q)\otimes_{\Q}\Q[t]\ \cong
\xymatrix{H^*_T(S,\Q)\ar[r]^(.48){\pi^*}_(.46)\sim & H^*_T(X,\Q).}
\end{equation}
Then for any integer $\alpha \geq 0$, define
\begin{equation} \label{defdesc}
\tau_\alpha(\sigma) := \pi_{P*} \big( \pi_X^{*} \sigma \cap \ch_{\alpha+2}^{T}(\FF) \big)\ \in\ H^{*}_{T}(P_X,\Q).
\end{equation}
The descendent invariants of $X$ are
\begin{equation*} 
P_{\chi,\beta}\big(X,\tau_{\alpha_1}(\sigma_1) \cdots \tau_{\alpha_m}(\sigma_m)\big) := \int_{[P_{\chi}(X,\beta)^{T}]^{\vir}} \frac{1}{e(N^{\vir})} \prod_{j=1}^{m} \tau_{\alpha_j}(\sigma_j)\big|\_{P_{\chi}(X,\beta)^{T}}
\end{equation*}
in $\Q[t,t^{-1}]$.


\section{The cosection} \label{cosecsec}

Let $\theta \in H^0(K_S)$ be a nonzero holomorphic 2-form with zero divisor $C$. We construct a natural induced cosection of the obstruction sheaf $\big(\!\Ob_{X}\!|_{P_{X}^{T}} \big)^f$. To use Serre duality it is convenient to compactify $X$,
$$
X\subset\Xb:=\PP(K_S\oplus\O_S),
$$
and use the projections 
$$
\xymatrix@=15pt{
& P_X \times\Xb \ar[dl]_{\overline\pi_P} \ar[dr]^{\pi_{\Xb}} & \\
P_X & & \Xb.}
$$
The universal stable pair $\II\udot = \{\mathcal{O}_{P_X\times X} \to \FF \}$ pushes forward to a universal stable pair
$$
\II\udot_{\Xb} := \big\{\mathcal{O}_{P_X\times\Xb} \To j_*\FF \big\}
$$
on $P_X\times\Xb$. Since $\II\udot_{\Xb}$ is isomorphic to $\O$ away from the support of $\FF$ in $X$, and since $\omega_{\Xb}$ is also trivial on restriction to $X$, we see that
$$
R\hom\big(\II\udot_{\Xb},\II\udot_{\Xb}\otimes\omega_{\Xb}\big)\_0\ \cong\ R\hom\big(\II\udot_{\Xb},\II\udot_{\Xb}\big)\_0\ =\ 
j_{*\,}R\hom\big(\II\udot,\II\udot\big)\_0\,.
$$
Pushing down by $\overline\pi_P$ gives
\beq{XXb}
R\hom_{\overline\pi_P}\big(\II\udot_{\Xb},\II\udot_{\Xb}\otimes\omega_{\Xb}\big)\_0\ \cong\ R\hom_{\overline\pi_P}\big(\II\udot_{\Xb},\II\udot_{\Xb}\big)\_0\ =\
R\hom_{\pi_P}\big(\II\udot,\II\udot\big)\_0\,.
\eeq

Translation by $\theta$ up the $K_S$ fibres of $X\to S$ defines a vector field $v_\theta$ on $X$, vanishing only on the preimage of $C \subset S$. Translating stable pairs by $v_\theta$ defines a vector field $V_\theta$ on $P_X$:
\beq{vfield}
V_\theta\ =\ v_\theta\ip\At(\II\udot)\ \in\
\Gamma\big(\ext^1_{\pi_P}\big(\II\udot,\II\udot\big)\_0\big).
\eeq
Pairing with the obstruction sheaf using \eqref{XXb} defines a map
\begin{eqnarray}
\ext^2_{\pi_P}\big(\II\udot,\II\udot\big)\_0 &\Rt{V_\theta\otimes1}&
\ext^1_{\pi_P}\big(\II\udot,\II\udot\big)\_0\otimes
\ext^2_{\pi_P}\big(\II\udot,\II\udot\big)\_0 \nonumber \\ &\cong& 
\ext^1_{\overline\pi_P}\big(\II\udot_{\Xb},\II\udot_{\Xb}\otimes\omega_{\Xb}\big)\_0\otimes
\ext^2_{\overline\pi_P}\big(\II\udot_{\Xb},\II\udot_{\Xb}\big)\_0 \label{Serre} \\ &\Rt{\cup}&
\ext^3_{\overline\pi_P}\big(\II\udot_{\Xb},\II\udot_{\Xb}\otimes\omega_{\Xb}\big)
\Rt{\tr}R^3\overline\pi_{P*}\omega_{\Xb}\ \cong\ \O_{P_X}. \nonumber
\end{eqnarray}
In the last Section we localised to the fixed locus $P^T_X\subset P_X$. Restricting \eqref{Serre} to $P^T_X$ and taking fixed (weight 0) parts gives a \emph{cosection}
\beq{cosec}
\sigma_\theta\colon\left(\!\Ob_{X}\!\big|\_{P^T_X}\right)^f\To\O_{P^T_X}.
\eeq
Its zero locus inherits a scheme structure from the cokernel of \eqref{cosec}.

Basechange issues\footnote{$\ext^1_{\pi_P}$ does not basechange well, but we will be able to use the fact that $\ext^2_{\pi_P}$ does.} make it nontrivial to equate the zero scheme of the cosection $\sigma_\theta$ with the zero scheme of vector field $V_\theta$ \eqref{vfield}. The correct formulation involves restricting $V_\theta$ to any subscheme $Z\subset P^T_X$ by first taking its image in the sheaf $\ext^1_{\pi_P}(\II\udot,\II\udot)_0|_Z$ then further restricting to $\ext^1_{\pi^Z_P}(\II\udot|_{Z\times X},\II\udot|_{Z\times X})_0$, where $\pi^Z_P\colon Z\times X\to Z$ is the restriction of $\pi_P\colon P_X\times X\to P_X$. Equivalently, but more directly, we just set
\beq{zfield}
V_{\theta,Z}\ :=\ v_\theta\ip\At(\II\udot|_{Z\times X})\ \in\ 
\Gamma\big(\ext^1_{\pi_P^Z}\big(\II\udot|_{Z\times X},\II\udot|_{Z\times X}\big)\_0\big).
\eeq
It is a $P_X$-vector field on $Z$ (so it need not be tangent to $Z$).

\begin{lemma} \label{20120007231089}
The zero locus $Z(\sigma_\theta)$ of the cosection \eqref{cosec} is the largest subscheme $Z\subset P_X^T$ for which $V_{\theta,Z}$ \eqref{zfield} is identically zero.
\end{lemma}

\begin{proof}
By basechange and the vanishing of the higher $(\ext_{\pi_P})_{0\,}$s, we have
\beq{OBZ}
\Big(\!\Ob_X\big|_Z\Big)^f\,=\,\Big(\ext^2_{\pi_P^Z}(\II\udot|_{Z\times X},\II\udot|_{Z\times X})_0\Big)^f,
\eeq
and the restriction of the cosection \eqref{cosec} to $Z$ is the map
\beq{coZ}
\Ob_X\big|_Z^f\To\O_Z
\eeq
given by restricting \eqref{Serre} to $Z$. It follows that the zero locus $Z(\sigma_\theta)$ is the largest $Z$ for which this map vanishes.

The map \eqref{coZ} is therefore the pairing with the section $V_{\theta,Z}$ \eqref{zfield} of
\beq{dZ}
\Big(\ext^1_{\pi_P^Z}\big(\II\udot|_{Z\times X},\II\udot|_{Z\times X}\otimes\omega_X\big)\_0\Big)^f.
\eeq
[Though $V_\theta$ has $T$-weight 1, the identification in the second line of \eqref{Serre} multiplies by the weight $-1$ trivialisation of $\omega_{\overline X}|\_X$, giving a $T$-fixed section.]
But this pairing makes the coherent sheaf \eqref{dZ} the dual $\hom(\Ob_X\!|_Z^f,\O_Z)$ of the sheaf $\Ob_X\!|_Z^f$ \eqref{OBZ}, by relative Serre duality for the map $\overline\pi_P^Z$, its compatibility with the $T$-action, and the vanishing of the other $\ext_{0\,}$s. Therefore $Z(\sigma_\theta)$ is the largest $Z\subset P_X^T$ for which the section $V_{\theta,Z}$ vanishes, as claimed.
\end{proof}

From now on we assume $C$ is reduced and irreducible. To describe a subscheme of $P^T_X$ containing the zero scheme of the cosection \eqref{cosec} we need some notation. For any (finite, 2-dimensional) partition $\lam = (\lambda_0 \geq \lambda_1 \geq \cdots\ge\lambda_{l-1})$, we denote by $\lam C \subset X$ the Cohen-Macaulay curve defined by the $T$-invariant ideal sheaf
$$
I_{\lam C}\,:=\ \O(-\lambda_0S)\ +\ I_C(-\lambda_1S)\ +\ I_C^2(-\lambda_2S)\ +\ \ldots\ +\ I_C^{l-1}(-\lambda_{l-1}S)\ +\,I_C^l.
$$
Here $I_C=\pi^*\O_S(-C)$ is the ideal sheaf of $\pi^*C$, and $\O(-S)\cong K_S^{-1}\otimes\t^{-1}$ is the ideal sheaf of the zero section $S\subset K_S=X$.

\begin{example} \label{vtr} The partitions $\lam=(4,2,1)$ and $\lam=(3,3,1)$ of 7 give the following two thickenings $\lam C$ of total size $|\lam|=\sum\lambda_i=7$. \smallskip
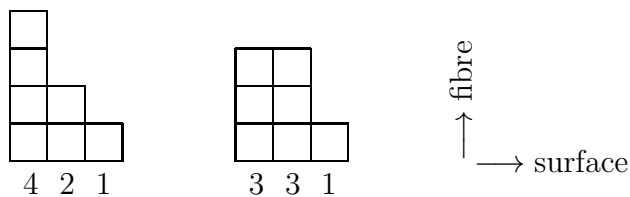
\begin{displaymath}
\xy
(0,0)*{}; (0,20)*{} **\dir{-} ; (5,0)*{} ; (5,20)*{} **\dir{-} ; (10,0)*{} ; (10,10)*{} **\dir{-} ; (15,0)*{} ; (15,5)*{} **\dir{-} ;  (0,0)*{} ; (15,0)*{} **\dir{-} ; (0,5)*{} ; (15,5)*{} **\dir{-} ; (0,10)*{} ; (10,10)*{} **\dir{-} ; (0,15)*{} ; (5,15)*{} **\dir{-} ; (0,20)*{} ; (5,20)*{} **\dir{-}; (8,-3)*{4\ \ 2\ \ 1\,};
(30,0)*{}; (30,15)*{} **\dir{-} ; (35,0)*{} ; (35,15)*{} **\dir{-} ; (40,0)*{} ; (40,15)*{} **\dir{-} ; (45,0)*{} ; (45,5)*{} **\dir{-} ;  (30,0)*{} ; (45,0)*{} **\dir{-} ; (30,5)*{} ; (45,5)*{} **\dir{-} ; (30,10)*{} ; (40,10)*{} **\dir{-} ; (30,15)*{} ; (40,15)*{} **\dir{-}; (38,-3)*{3\ \ 3\ \ 1\,};
(72,0)*{\longrightarrow \mathrm{surface}} ; (60,8)*{\rotatebox{90}{$\longrightarrow \mathrm{fibre}$}}
\endxy
\end{displaymath}
The first is \emph{strict}: $\lambda_0>\lambda_1>\cdots>\lambda_{l-1}>0$, while the second is not. 
\end{example}

We can slice horizontally instead of vertically. If $\lam^t=(\mu_0,\mu_1,\cdots)$ denotes the transpose partition,
we can think of $\lam C$ as the curve obtained by thickening $C$ to order $\mu_0$ at $T$-weight level 0, $\mu_1$ at $T$-weight level $-1$, etc.:
$$
I_{\lam C}\ =\ I_C^{\mu_0}\,+\,I_C^{\mu_1}(-S)\,+\,I_C^{\mu_2}(-2S)\,+\,\ldots\,+\,I_C^{\mu_{k-1}}(-(k-1)S)\,+\,\O(-kS).
$$
In the above Example \ref{vtr}, the transposed partitions $\lam^t$ are $(3,2,1,1)$ and $(3,2,2)$ respectively.


If $\lam$ has size $|\lam|=\sum\lambda_i=d$ we write $\lam\vdash d$. We fix $\chi$ throughout this Section and denote by
$$
P_{\lam C} := P_{\chi}(\lam C)\subset P_{\chi}(X,d[C])
$$
the moduli space of stable pairs with holomorphic Euler characteristic $\chi$ whose scheme-theoretic support is precisely $\lam C$. Since $\lam C$ is $T$-invariant, $P_{\lam C}$ has a $T$-action and its fixed locus is a closed subscheme
$$
P_{\lam C}^{T}\ =\ P_{\lam C}\cap P_{\chi}(X,d[C])^{T}.
$$
We will find that the support of stable pairs in the zero locus $Z(\sigma_\theta)$ of the cosection have support $\lam C$ for $\lam$ \emph{strict}.

\begin{proposition} \label{deglocinside...}
The zero scheme $Z(\sigma_\theta)$ of the cosection \eqref{cosec} is nonempty only if $\beta = d[C]$ for some $d>0$. In this case, it is a closed subscheme of 
$$
\bigsqcup_{\lam\,\vdash d\ \mathrm{strict}}P_{\lam C\,}^{T}.
$$
\end{proposition}
\begin{proof}
Let $Z:=Z(\sigma_\theta)$ and let $s$ denote the tautological section of $\pi^*K_S$ cutting out the zero section $S\subset X$. We use $T$-invariance to write the ideal sheaf of the support of $\FF|_{Z\times X}$ in the form
\begin{align} 
\begin{split} \label{ideal0}
&\pi^* I_0 + \pi^* I_1 . s + \cdots + \pi^* I_{k-1} . s^{k-1} + (s^k), \\
&\qquad I_0 \subset I_1 \subset \cdots \subset I_{k-1} \subset \O_{Z \times S},
\end{split}
\end{align}
for some integer $k>0$.

Let $t$ denote the coordinate on $\C_t:=\C$. Then pulling back $\II\udot|\_{Z\times X}$ to $Z\times X\times\C_t$ and translating by $tv_\theta$ gives a new family of stable pairs over $Z\times X\times\C_t$ whose support is defined by the ideal
\beq{idealt}
\pi^* I_0 + \pi^* I_1 . (s-t\pi^*\theta) + \cdots + \pi^* I_{k-1} . (s-t\pi^*\theta)^{k-1} + ((s-t\pi^*\theta)^k).
\eeq
Restricting to $\Spec\C[t]/(t^2)\subset\C_t$ gives a flat family of stable pairs on $X$ parameterized by $Z\times\Spec\C[t]/(t^2)$ whose support has ideal \eqref{idealt} mod $t^2$,
\beq{moved}
\pi^* I_0 + \pi^* I_1 . (s-t\pi^*\theta) + \cdots + \pi^* I_{k-1} . (s^{k-1}-(k-1)t\pi^*\theta s^{k-2}) + (s^k-k t\pi^*\theta s^{k-1}).
\eeq
The corresponding first order deformation of $\II\udot|_{Z\times X}$ is classified by its extension class 
$$
V_{\theta,Z}\ \in\ \Ext^1\!\big(\II\udot|\_{Z\times X},\II\udot|\_{Z\times X}\big)\_0\ =\
\Gamma\big(\ext^1_{\pi_P|_Z}(\II\udot|\_{Z\times X},\II\udot|\_{Z\times X})\_0\big)
$$
of \eqref{zfield}. By Lemma \ref{20120007231089} this is zero, so the family is trivial. In particular its support is pulled back from $Z\times X$, so \eqref{moved} is the same ideal as \eqref{ideal0}$\otimes\C[t]/(t^2)$. That is,
$$
\theta\cdot I_i \subset I_{i-1}, \ \forall i = 1, \ldots, k-1, \qquad\mathrm{and}\ \theta \in I_{k-1}.
$$
Since $C$ is reduced and irreducible, and each $\O/I_i$ is pure, this implies that each $I_i=(\theta^{\mu_i})$ for some integer $\mu_i$, and that $\mu_i+1\ge\mu_{i-1}$. Thus we can write \eqref{ideal0} as
$$
(\theta^{\mu_0})+(s\theta^{\mu_1})+\cdots+(s^{k-1}\theta^{\mu_{k-1}})+(s^k),
$$
where we have suppressed some $\pi^*$s for clarity. Rewriting this as
$$
(s^{\lambda_0})+(\theta s^{\lambda_1})+\cdots+(\theta^{l-1}s^{\lambda_{l-1}})+(\theta^l),
$$
where $\lam=(\lambda_0,\lambda_1,\ldots)$ is the transpose of the partition $(\mu_0,\mu_1,\ldots)$, the condition $\mu_i+1\ge\mu_{i-1}$ becomes the requirement that $\lam$ be strict.
\end{proof}

So in Example \ref{vtr} we find that $P^T_{\lam C}$ contains zeros of the cosection when $\lam = (4,2,1)$, but not when $\lam = (3,3,1)$.

We have only considered the effect of the cosection on the underlying Cohen-Macaulay support curve of a stable pair, showing it forces it to be of the form $\lam C$ with $\lam$ strict. The proof also shows that bare curves of this form (i.e. a stable pair isomorphic to $(\O_{\lam C},1)$ with no cokernel of ``free points") lie in $Z(\sigma_\theta)$. For more general stable pairs, being in $Z(\sigma_\theta)$ also imposes conditions on its cokernel; see the sequel \cite{KT4} for more details.

In this paper we content ourselves with a characterization of \emph{vertical component} of $Z(\sigma_\theta)$,  where $\lam=(d)$ has length 1. Here there is no further condition on the cokernels of stable pairs.

\begin{corollary} \label{vercomp}
The zero scheme $Z(\sigma_\theta)$ of the cosection \eqref{cosec} on $P_{d[C]}^T$ has a component
$$
P_{(d)C}^{T} := P_{\lam C}^{T}, \quad \lam = (d).
$$
\end{corollary}

\begin{proof}
The vector field $v_{\theta}$ vanishes on $\pi^*C \subset X$. As a consequence the vector field $V_\theta$ vanishes on $P_{(d)C}^{T}$ which therefore lies in the zero scheme $Z$ of the cosection \eqref{cosec}. By Proposition \ref{deglocinside...} it is a whole component of $Z$. (In fact we will see in Proposition \ref{degloc=nested} it is a disjoint union of connected components.)
\end{proof}

\begin{corollary} \label{hot}
Assume $S$ has a reduced, irreducible canonical divisor $C$. Then
$$
P_{\chi,\beta}\big(X,\tau_{\alpha_1}(\sigma_1) \cdots \tau_{\alpha_m}(\sigma_m)\big)=0
$$
unless $\beta = d\k$ for some $d\in\Z_{>0}$ and all $\sigma_i$ lie in $H^{\le2}(S)$.
\end{corollary}

\begin{proof}
For $\beta$ not a multiple of $d\k$ then $Z(\sigma_\theta)$ is empty by Proposition \ref{deglocinside...}, so the invariants vanish. 

If $\sigma \in H^{\geq 3}(S)$, we can write $\sigma = [\gamma]$ for some cycle $\gamma \in H_{\leq 1}(S)$ disjoint from $C$. Therefore $\pi_X^{*}\sigma \cap \ch_{\alpha+2}^{T}(\FF)=0$ over the locus of pairs with support $\lam C$, so the insertions $\tau_\alpha(\sigma)$ certainly vanish over $P^T_{\lam C}$ for any strict $\lam \vdash d$. Since the virtual cycle can be cosection localised to this locus, the associated invariants vanish. This completes the proof of Theorem \ref{main0} in the Introduction. 
\end{proof}

\section{Nested Hilbert schemes} \label{sectionneshilb}

We now begin the process of describing $T$-fixed stable pairs --- especially those in the vertical component $P^T_{(d)C}$ of $Z(\sigma_\theta)$ --- more explicitly.

\subsection{$T$-equivariant sheaves on $X$}
Given a $T$-equivariant coherent sheaf $F$ on $X$, its pushdown by $\pi\colon
X\to S$ decomposes into weight spaces:
\beq{1}
\pi_*F\ =\ \bigoplus_iF_i\otimes\t^i,
\eeq
where $F_i$ is $T$-fixed so $F_i\otimes\t^i$ is the summand of weight $i$. For instance
\beq{OX}
\pi_*\O_X\ =\ \bigoplus_{i\ge0}K_S^{-i} \otimes \mathfrak{t}^{-i}.
\eeq
Since $\pi$ is affine, the pushdown loses no information; we can recover the
$\O_X$-module structure on $F$ by describing the action of \eqref{OX} that
\eqref{1} carries. This is generated by the action of the weight $-1$ piece $K_S^{-1}\otimes\t^{-1}$, so we find
that the $\O_X$-module structure is determined by the map
\beq{2}
\bigoplus_iF_i\otimes \mathfrak{t}^i \otimes \big( K_S^{-1} \otimes \mathfrak{t}^{-1} \big) \To\bigoplus_iF_i \otimes \mathfrak{t}^i,
\eeq
which commutes with both the actions of $\O_S$ and $T$. That is, \eqref{2} is a
$T$-equivariant map of $\O_S$-modules. By $T$-equivariance, it is a sum of
maps
\beq{3}
F_i\otimes K_S^{-1}\To F_{i-1}.
\eeq

\subsection{$T$-equivariant pairs on $X$} \label{Tpair}
Having described $T$-equivariant coherent sheaves $F$ on $X$ as graded sheaves
\eqref{1} on $S$ with $T$-equivariant maps \eqref{3}, we can generate a similar
description of $T$-equivariant pairs $(F,s)$ on $X$. Here $s\in H^0(F)^T$ is a
$T$-equivariant section of $F$.

Applying $\pi_*$ to $\O_X\rt{s}F$ gives a graded map between \eqref{OX} and
$$
\xymatrix@C=0pt@R=14pt{
&&&\O_S\dto&\oplus&\big(K_S^{-1} \otimes \mathfrak{t}^{-1} \big) \dto&\oplus&\big(K_S^{-2} \otimes \mathfrak{t}^{-2} \big)\dto&\oplus\ \cdots\ \\
\cdots&F_1\otimes \mathfrak{t} &\oplus&F_0&\oplus& \big(F_{-1} \otimes \mathfrak{t}^{-1} \big) &\oplus& \big(F_{-2} \otimes \mathfrak{t}^{-2} \big)&\oplus\ \cdots,\!}
$$
which commutes with the maps \eqref{3} along the top and bottom rows.
So writing
\beq{pair2}
G_i:=F_{-i}\otimes K^i_S\,,
\eeq
(which is $T$-fixed) we find the data $(F,s)$ on $X$ is equivalent to the
following data of sheaves and commuting maps on $S$:
\beq{pairs2}
\xymatrix@C=14pt@R=14pt{
&\O_S\dto\ar@{=}[r]&\O_S\dto\ar@{=}[r]&\O_S\dto\ar@{=}[r]&\cdots \\ \cdots
G_{-1}\rto&G_0\rto&G_1\rto&G_2\rto&\cdots\ .\!\!\!}
\eeq

\subsection{$T$-equivariant stable pairs in the vertical component} \label{that}
In Section \ref{Tpair} we gave a general description of $T$-equivariant pairs on $X$. Now we restrict attention to $T$-equivariant \emph{stable} pairs $(F,s)$ whose scheme
theoretic support is $\pi^{-1}C$ for some fixed \emph{connected smooth} curve
$C\subset S$. This will lead to a description of the connected component $P^T_{(d)C}$ of $Z(\sigma_\theta)$ of Corollary \ref{vercomp}. We only consider pairs with proper support, which implies that
there is a maximal $d\ge0$ such that $G_{d-1}\ne0$ in the description
\eqref{pair2}. (This is the smallest $d$ such that $F$ is supported on
$dS\subset X$.)

Thus $F$ is pushed forward from $\pi^{-1}(C)$ and $\O_X\rt{s}F$ has finite
cokernel. Thus all of the sheaves $G_i$ in \eqref{pairs2} are supported on $C$, the
vertical maps factor through $\O_S\to\O_C$, and generically on $C$ the induced
maps from $\O_C$ are isomorphisms. It follows in particular that $G_{-i}$ is
0-dimensional for $i>0$ and so vanishes by purity of $F$.

The upshot is that the stable pair is equivalent to a commutative diagram
\beq{penult}
\xymatrix@C=14pt@R=14pt{
\O_C\dto\ar@{=}[r]&\O_C\dto\ar@{=}[r]&\O_C\dto\ar@{=}[r]&\cdots\ar@{=}[r]&\O_C\dto
\\ G_0\rto&G_1\rto&G_2\rto&\cdots\rto&G_{d-1}\!\!}
\eeq
of $\O_C$-modules, with each $G_n$ pure 1-dimensional and each vertical map an
isomorphism away from a finite number of points.

Since $C$ is smooth, it follows that each $G_i$ is a line bundle with section, that the horizontal
maps are all injections, and the diagram is the top two rows of
\beq{final}
\xymatrix@C=14pt@R=14pt{
\O_C\dto\ar@{=}[r]&\O_C\dto\ar@{=}[r]&\O_C\dto\ar@{=}[r]&\cdots\ar@{=}[r]&\O_C\dto
\\ \O(Z_0)\dto\INTO&\O(Z_1)\dto\INTO&\O(Z_2)\dto\INTO&\cdots\
\INTO&\O(Z_{d-1})\!\!\dto \\
\O_{Z_0}(Z_0)\INTO&\O_{Z_1}(Z_1)\INTO&\O_{Z_2}(Z_2)\INTO&\cdots\
\INTO&\O_{Z_{d-1}}(Z_{d-1}).\!\!\!}
\eeq
Here the $Z_i$ are Cartier divisors on $C$, and all columns are the obvious short exact sequences.

\subsection{Stable pairs and the nested Hilbert scheme}
Thus a $T$-equivariant stable pair $(F,s)$ with proper support in $\pi^{-1}(C)$
is equivalent to a chain of divisors
\beq{nesting}
Z_0\subset Z_1\subset Z_2\subset\cdots\subset Z_{d-1}\subset C.
\eeq
Hence it defines a point of the nested Hilbert scheme
$$
C^{[\n]}, \qquad \n=(n_0, \ldots, n_{d-1}),
$$
of length-$n_i$ zero-dimensional subschemes $Z_i$ of $C$ satisfying the nesting
condition \eqref{nesting}. Here
$$
\chi(F)\ =\ \sum_{i\ge0}\chi(F_{-i})\ =\ \sum_{i\ge0}\chi(G_i\otimes
K_S^{-i}|_C)
\ =\ \sum_{i\ge0}\big(\chi(K_S^{-i}|_C)+n_i\big)
$$
determines $|\n|=\sum n_i$. When $C\in|K_S|$ is a canonical
curve, we find that
$$
\chi(F)=\sum_i\big(n_i-(i+1)\k^2\big),
$$
where $\k:=c_1(K_S)$.

Conversely, a point of the nested Hilbert scheme gives a diagram \eqref{final},
which we have noted is equivalent to a $T$-fixed stable pair on $X$ supported
on $\pi^{-1}(C)\cap dS$. Thus we get a set-theoretic isomorphism
\beq{isomorphism}
P^T_{\chi,(d)C}\ =\ \bigsqcup_{\n} C^{[\n]},
\eeq
where the disjoint union is taken over all $\n=(n_0, \ldots, n_{d-1})$ whose
length $|\n|$ satisfies 
\beq{chi/n}
\chi=\sum_i\big(n_i-(i+1)\k^2\big).
\eeq
\begin{proposition} \label{degloc=nested}
The bijection \eqref{isomorphism} is an isomorphism of schemes.
\end{proposition}

\begin{proof}
We simply notice that the constructions of this Section work equally well for $T$-equivariant sheaves and stable pairs on $X\times B$, flat over any base $B$.

Pushing down by the affine map $\pi\colon X\times B\to S\times B$ gives a
graded sheaf $\bigoplus_iF_i$ on $S\times B$. It is flat over $B$, therefore so
are all its weight spaces $F_i$. The original sheaf $F$ on $X\times B$ can be
reconstructed from the maps \eqref{3}.
Therefore a $T$-equivariant stable pair $(F,s)$ on $X\times B$, flat over $B$,
is equivalent to the data \eqref{pairs2} with each $G_i$ flat over $B$.

When $F$ is supported on $\pi^{-1}(C\times B)$, with $C$ a smooth connected
curve in $S$, we showed that each $G_i$ is a line bundle on any closed fibre
$C\times\{b\}$ (where $b\in B$). Being locally free is an open condition on
sheaves, so this shows that each $G_i$ is a line bundle on $C\times B$.
Together with its nonzero section \eqref{pairs2} we find it defines a divisor
$Z_i\subset C\times B$, flat over $B$.

Thus we get the diagram \eqref{final} of flat sheaves and nested divisors over
$B$.
This defines a classifying morphism $B\to\bigsqcup_{\n} C^{[\n]}$.

Conversely, the universal family on $C^{[\n]}$ defines a diagram
\eqref{final}, equivalent to a $T$-equivariant stable pair $(F,s)$ on $X\times
B$ supported on $$(\pi^{-1}(C)\cap dS)\times B$$ and flat over $B$. This
defines the inverse classifying map $\bigsqcup_{\n} C^{[\n]} \to B$.
\end{proof}

\subsection{The dual description}
In this Section we give an explicit description of the pairs constructed in the last Section in terms of the geometry of the vertical thickening $(d)C\subset dS\subset X$. For clarity of exposition we work at a single point of moduli space, though just as in the last Section there is no difficulty in having everything vary in a flat family over a base $B$.

So we fix a point of $C^{[\n]}$, i.e. an increasing flag of effective divisors
$$
Z_0\subset Z_1\subset Z_2\subset\cdots\subset Z_{d-1}\subset C \qquad
$$
as in \eqref{nesting}. Setting $D_i:=Z_{d-1}-Z_i$ gives a dual \emph{decreasing} flag of effective divisors
\beq{DDd}
\qquad\quad D_0\supset D_1\supset D_2\supset\cdots\supset D_{d-2}, \qquad D_{d-1}=\varnothing,
\eeq
in $C$. These fit together to define a \!\emph{Weil divisor}\footnote{$D\subset (d)C$ is Cartier if and only if all the $D_i$ are empty.} $$D\subset (d)C$$ in the way described in Section \ref{Tpair}. That is, take $G_i=\O_{D_i}$ in \eqref{pairs2} and use the following example of the diagram \eqref{penult}, $$
\xymatrix@C=16pt@R=14pt{
\O_S\dto<-.5ex>\ar@{=}[r]&\O_S\dto<-.5ex>\ar@{=}[r]&\O_S\dto<-.5ex>\ar@{=}[r]&
\cdots\ar@{=}[r]&\O_S\dto<-.5ex>\ar@{=}[r]&\O_S\dto<-.3ex> \\
\O_{D_0}\!\!\rto&\O_{D_1}\!\!\rto&\O_{D_2}\!\!\rto&\cdots\rto&\O_{D_{d-2}}\!\!\rto&0.\!}
$$
All arrows are the obvious restriction maps. By the construction of Section \ref{Tpair} this is equivalent to a $T$-equivariant pair $\O_X\to G$ with no cokernel, so $G$ must be a structure sheaf $\O_D$ of a subscheme $D\subset (d)C$ such that $\pi_*\O_D$ is $$\bigoplus_{i=0}^{d-2}\O_{D_i}\otimes K_S^{-i} \otimes \t^{-i}.$$

Now $\pi^*Z_{d-1}$ is a Cartier divisor on $(d)C$, defining a line bundle $\O_{(d)C}(\pi^*Z_{d-1})$ with a canonical section vanishing on $\pi^*Z_{d-1}\supset D$. It therefore factors through the ideal sheaf $I_D$ of $D\subset (d)C$, defining a unique section
\beq{this}
\O_X\Rt{s}\O_{(d)C}(\pi^*Z_{d-1})\otimes I_D.
\eeq
This defines a $T$-equivariant stable pair.

\begin{proposition} The isomorphism of Proposition \ref{degloc=nested} takes the nested flag of subschemes $Z_0\subset Z_1\subset Z_2\subset\cdots\subset Z_{d-1}\subset C$ to the stable pair \eqref{this}.
\end{proposition}

\begin{proof}
By Section \ref{that}, the stable pair \eqref{this} is described by a diagram of the form \eqref{penult}. By pushing down \eqref{this} we find that
$$
G_i\ =\ \O_C(Z_{d-1})\otimes I_{D_i},
$$
which by the definition of $D_i$ is
$$
\O_C(Z_{d-1}-D_i)\ \cong\ \O_C(Z_i).
$$
Therefore for the pair \eqref{this}, the diagram \eqref{penult} becomes
$$
\xymatrix@C=14pt@R=14pt{
\O_C\dto\ar@{=}[r]&\O_C\dto\ar@{=}[r]&\O_C\dto\ar@{=}[r]&\cdots\ar@{=}[r]&\O_C\dto
\\ \O(Z_0)\INTO&\O(Z_1)\INTO&\O(Z_2)\INTO&\cdots\
\INTO&\O(Z_{d-1}),\!\!}
$$
with all maps the canonical ones. But this is precisely the diagram \eqref{final} corresponding to the flag $Z_0\subset Z_1\subset Z_2\subset\cdots\subset Z_{d-1}\subset C$ from which we construct the $T$-equivariant stable pair via the isomorphism \eqref{isomorphism}.
\end{proof}

\begin{remark} This description of stable pairs in terms of Hilbert schemes parameterising either the subschemes $Z_i$ \eqref{nesting} or the dual subschemes $D_i$ \eqref{DDd} is related to, but \emph{different from}, the description \cite[Appendix B.2]{PT3} of stable pairs on surfaces in terms of relative Hilbert schemes. The latter description is dual to the one above in a different way, involving the (derived dual) of the sheaf $F$ and complex $I\udot=\{\O_S\to F\}$. 
\end{remark}

\section{Localised virtual cycle} \label{virtsec}


In Corollary \ref{vercomp} we showed that the contribution to $[P_{d\k}^{T}]^{\vir}$ of its vertical component is the push forward of a cycle on $P_{(d)C}^{T}\cong C^{[\n]}$. We denote this Kiem-Li \cite{KL3} cosection localised virtual cycle by
\beq{KLvir}
\big[P^T_{\ver}\big]^{\vir}\in A_*(C^{[\n]}).
\eeq
In this Section we compute it. \medskip

Denote by $$\Hk := \Hilb_{\k}(S)$$ the Hilbert scheme of effective divisors in class $\k=c_1(K_S)$ on $S$. A result of H.-l.~Chang and Y.-H.~Kiem \cite{CK} simplifies our life considerably.

\begin{theorem} \label{ChKi}
Assume that $S$ has a smooth irreducible canonical divisor $C$. Then we may assume $C$ defines a \emph{smooth point} of $\Hk$ at which $$\dim_{\{C\}\!}\Hk\equiv\chi(\O_S) \ \mod 2.$$
\end{theorem}

\begin{proof}
Chang-Kiem \cite[Proposition 4.2]{CK} use a result of Green-Lazarsfeld to prove that there exists a canonical divisor at which $\Hk$ is smooth. It follows that the smooth locus of $\Hk$ intersects $|K_S|$ in a nonempty Zariski open subset.

The smooth irreducible canonical divisors form another Zariski open subset of $|K_S|$, and our assumption implies it is also nonempty. Since $|K_S|$ is a projective space it is in particular irreducible, so the two Zariski open subsets have nonempty intersection. Choosing $C$ in this intersection gives the result.

Finally the parity of $\dim\Hk$ at $C$ is also given in \cite[Proposition 4.2]{CK}.
\end{proof}

Using this result, we will find we are in the following situation.

Consider $M$ a projective scheme with perfect obstruction theory, obstruction sheaf $\Ob$ and cosection vanishing on $Z\stackrel\iota\Into M$:
$$\Ob\Rt{\sigma}\O_M\To\iota_*\O_Z\To0.$$
Suppose that $Z$ is smooth, and that $M$ is smooth in a neighbourhood of $Z$. It is then clear what the virtual cycle of $M$ should be. Away from $Z$ the surjection $\Ob\to\O_M$ ensures that it is zero. Fulton-MacPherson intersection theory allows us to write it as the pushforward of a class on $Z$ which, by smoothness and the locally freeness of $\Ob$ near $Z$, should calculate $c_{\mathrm{top}}(\Ob)$. To find it we use the exact sequence
\beq{exsq}
0\To T_Z\To T_M|_Z\Rt{d\sigma|_Z}\Ob^*\!|_Z\To Q\To0
\eeq
on $Z$. Here $Q\cong\Ob^*\!|_Z\big/N_{Z/M}$ is defined to be the cokernel of $d\sigma|_Z$; this is locally free by the smoothness of $Z\subset M$. Excess intersection theory says that its top Chern class on $Z$, pushed forward to $M$, represents the top Chern class of $\Ob^*$:
\beq{Qob}
\iota_*c_{\mathrm{top}}(Q)=c_{\mathrm{top}}(\Ob^*).
\eeq
Let $m$ denote the dimension of $M$ in the neighbourhood of $Z$, and let $vd$ be the virtual dimension of the obstruction theory. Therefore $r:=\rk(\Ob\!|_Z)$ is $m-vd$. Finally let $c$ denote the codimension of $Z\subset M$, so that $\rk(Q)=r-c=m-vd-c$.

Since \eqref{Qob} differs from $c_{\mathrm{top}}(\Ob)$ only by the sign $(-1)^r$, and we expect the virtual cycle to be
$$
(-1)^r\iota_*\big(c_{\mathrm{top}}(Q)\big)\,=\ (-1)^{m-vd}\iota_*\big(c_{m-vd-c}(Q)\big)\,=\ (-1)^c\iota_*\big(c_{m-vd-c}(Q^*)\big).
$$
By \eqref{exsq} this is
$$
(-1)^c\big[\iota_*c(Q^*)\big]_{vd}\ =\ (-1)^{c\,}\iota_*\big[c(\Ob\!|_Z)s(N^*_{Z/M})\big]_{vd}\,,
$$
where $c(\,\cdot\,)$ and $s(\,\cdot\,)$ denote the \emph{total} Chern and Segre classes respectively. Unsurprisingly, the formulation of Kiem-Li gives precisely this answer.

\begin{proposition} \label{gencaseZM}
In the above situation, Kiem and Li's localised virtual cycle of $M$ is the class in $A_{vd}(Z)$ given by the $vd$-dimensional part of $$(-1)^c\big(c(\Ob\!|_Z)s(N^*_{Z/M})\big)\cap[Z].$$
\end{proposition}

\begin{proof} 
In our situation Kiem and Li's recipe for their localised class is the following. Let
$$\xymatrix@=18pt{
E\subset\Bl_ZM\quad \ar[d]^p \\ M}
$$
be the blow up of $M$ in $Z$ with exceptional divisor $E$. Then the pullback of the cosection has zero locus $E$, giving an exact sequence
$$
0\To G\To p^*\Ob\Rt{p^*\sigma}\O(-E)\To0
$$
for some vector bundle $G$ of rank $g=r-1=m-vd-1$. Kiem and Li tell us to intersect the zero section of $G$ with itself and then with $-E$, and push the result down to $Z$. Since $p|_E\colon E\to Z$ is the projective bundle $\PP(N_{Z/M})\to Z$ of relative dimension $c-1$, this gives
\beqa
-(p|_E)\_{*\,}c_{\mathrm{top}}(G|_E) &=& -\Big[(p|_E)\_{*\,}c(G|_E)\Big]_{m-1-g} \\
&=& -\Big[(p|_E)\_{*}\big(p^*c(\Ob\!|_Z)s(\O_E(-E))\big)\Big]_{vd} \\
&=& -\Big[c(\Ob\!|_Z)\cdot(p|_E)\_{*\,}s(\O_{\PP(N_{Z/M})}(1))\Big]_{vd} \\ &=& -\Big[c(\Ob\!|_Z)\cdot(-1)^{c-1}s(N^*_{Z/M})\Big]_{vd}\ ,
\eeqa
which gives the required result.
\end{proof}

We can apply this to describe the virtual cycle $[P_{\ver}^{T}]^{\vir}$ \eqref{KLvir} as follows. Recall that the zero locus of cosection \eqref{cosec} is
$$
\bigsqcup_{\n} C^{[\n]}\,,
$$
where the sum is over all $\n$ satisfying \eqref{chi/n}. Note that $n_{d-1}$ is the length $l(Z_{d-1})$ of the last divisor in the flag \eqref{nesting} --- i.e. the dimension of the nested Hilbert scheme $C^{[\n]}$.

\begin{corollary} \label{virtualcycle}
Under the assumptions of Theorem \ref{ChKi}, the Kiem-Li cosection-localised virtual cycle of the connected component $C^{[\n]}$ of $P_{(d)C}^{T}$ is
\beq{530}
(-1)^{\chi(\O_S)}\cdot c_{n\_{d\!-\!1\!}-vd}\Big(\!\Ob\!\big|\_{C^{[\n]}}\!\Big)\ \in\, A_{vd}(C^{[\n]}).
\eeq
Therefore $[P_{\ver}^{T}]^{\vir}$ is the sum of (the pushforwards of) these classes over all nonnegative integers $n_0 \leq \cdots \leq n_{d-1}$ satisfying
\beq{neqn2}
\sum_{i=0}^{d-1}(n_i - (i+1)\k^2) = \chi.
\eeq
\end{corollary}

\begin{remark} \label{rmkvfc} We will see in \eqref{virtdim2} below that $vd=n_0$, so in the \emph{uniformly thickened} case $n_0=\ldots=n_{d-1}$ the localised virtual class is just $(-1)^{\chi(\O_S)}[P^T_{(d)C}]$.
\end{remark}

\begin{proof}
Let $U$ denote the smooth Zariski open neighbourhood $U\subset\Hk$ of the smooth point $\{C\}$ given to us by Theorem \ref{ChKi}. The nested Hilbert scheme of the (smooth!) universal curve over $U$ defines a neighbourhood of $P^T_{(d)C}\subset P^T_{X}$:
\beq{Unhd}
P_U:=P_{X}^T\big|\_U\,.
\eeq
By the same working as in Section \ref{sectionneshilb} this is isomorphic to the open set of $P_{X}^T$ consisting of stable pairs supported on curves in $U$.

Since the nested Hilbert schemes of smooth curves are smooth, $P_U\to U$ is a smooth map. Therefore both $P^T_{(d)C}$ and $P_U$ are smooth, and by Proposition \ref{vercomp} we can apply Proposition \ref{gencaseZM} to $P^T_{(d)C}\subset P^T_{X}$ in place of $Z\subset M$.

Since $N^*_{Z/M}$ is the pullback of the conormal bundle of $\{C\}\subset H_k$, it is trivial on $P^T_{(d)C}$ with Segre class 1. And the codimension of $Z\subset M$ is $c=\dim_{\{C\}}\Hk\equiv\chi(\O_S)$ mod 2, which fixes the sign. Finally, the sum is over $\n$ satisfying \eqref{chi/n}.
\end{proof}

Therefore to compute we need only calculate the K-theory class of the bundle
$$
\Ob\!\big|\_{C^{[\n]}}\,=\,\ext_{\pi_P}^{2}(\II\udot, \II\udot)_{0}\big|_{C^{[\n]}}^f.
$$
We do this in the next Section. This will also determine the value of $vd$ (which we have not yet found, notice!).

\section{Obstruction bundle} \label{vircycleII}

Throughout this Section we use the notation
$$
\langle\ \cdot\ ,\ \cdot\ \rangle:=\big[R\hom_{\pi_P}(\,\cdot\,,\,\cdot\,)\big],
$$
where the square brackets take the $T$-equivariant K-theory class of an element of the equivariant derived category $D(P^T_X)^T$. We will compute
the restriction to $P^T_{(d)C}$ of the (dual of the) perfect obstruction theory \eqref{pofs}:
$$
\big[E\udot\big]^\vee\ =\ -\langle\II\udot,\II\udot\rangle\_{0\,}.
$$
As usual the subscript denotes the trace-free part.

We work on the neighbourhood $P_U$ \eqref{Unhd} of $P^T_{(d)C}\subset P^T_{X}$. Thus we have the description of Section \ref{sectionneshilb}, which we now summarise. $U\subset\Hk$ is a smooth open set of smooth curves in class $\k=c_1(K_S)$ with universal curve
$$
\cC\Rt{p}U
$$
whose relative nested Hilbert scheme is isomorphic to $P_U$:
$$\xymatrix@R=0pt{
\Hilb^\n(\cC) \ar[dd]^p && P_U \ar[dd]^p \\ &\cong \\
U && U.\!}
$$
That is, over $P_U$ the universal curve carries a universal family of nested divisors
$$
\cZ_0\subset\cZ_1\subset\cZ_2\subset\cdots\subset\cZ_{d-1}\subset p^*\cC\ \subset\ X\times P_U.
$$
These define the universal stable pair via \eqref{final} (or equivalently via \eqref{this}).

Therefore the universal sheaf $\FF$ is an iterated (and equivariant) extension of the sheaves
\beq{amic}
\O_\cC(\cZ_i)\otimes K_S^{-i},\quad i=0,\ldots,d-1,
\eeq
on $X\times P_U$. (The sheaves \eqref{amic} are of course pushed forward from $S\times P_U$; they are the eigensheaves of the $T$-action on $\pi_*\FF$ as in Section \ref{sectionneshilb}.) Hence the K-theory class of the universal sheaf is
\beq{decom}
\big[\FF\big]=\sum_{i=0}^{d-1}\big[\O_\cC(\cZ_i)\otimes K_S^{-i} \otimes \t^{-i}\big].
\eeq
Similarly the class of the universal complex is
$$
\big[\II\udot\big]=\big[\O_{X\times P_U}\big]-\big[\FF\big],
$$
from which we compute
\beqa
-\langle\II\udot,\II\udot\rangle\_0 &=& \langle\O_{X\times P_U},\FF\rangle
+\langle\FF,\O_{X\times P_U}\rangle-\langle\FF,\FF\rangle \\
&=& \big[R\pi_{P*}\FF\big]-\big[R\pi_{P*}\FF\big]^\vee\!\otimes\t
-\langle\FF,\FF\rangle
\eeqa
by ($T$-equivariant) Serre duality. By \eqref{decom} this is
\beq{Nvir2}
-\langle\II\udot,\II\udot\rangle\_0\ =\ \sum_{i=0}^{d-1}
\big[R\pi_{P*}\big(\O_\cC(\cZ_i)\otimes K_S^{-i}\big)\big]\t^{-i}
-\big[R\pi_{P*}\big(\O_\cC(\cZ_i)\otimes K_S^{-i}\big)\big]^\vee\t^{i+1} -\langle\FF,\FF\rangle,
\eeq
where
\beq{homFF}
-\langle\FF,\FF\rangle\ =\ -\sum_{i,j=0}^{d-1}
\langle\O_\cC(\cZ_i),\O_\cC(\cZ_j)\otimes K_S^{i-j}\rangle\t^{i-j}.
\eeq
Since $R\hom(\O_\cC(\cZ_i),\O_\cC(\cZ_j))$ has the same K-theory class as the alternating sum of its cohomology sheaves, a local Koszul resolution gives
$$
\big[R\hom(\O_\cC(\cZ_i),\O_\cC(\cZ_j))\big]=\big[\big(\O_\cC-\O_\cC(\cC)-K_S\big|\_\cC\t
+K_S(\cC)\big|\_\cC\t\big)(\cZ_j-\cZ_i)\big].
$$
Substituting into \eqref{homFF}, we find
\begin{align}
-\langle\FF,\FF\rangle\ =\ R\pi_{P*}\sum_{i,j=0}^{d-1}\bigg[&K_S^{i-j+1}
\big|\_\cC(\Delta_{ij})\t^{i-j}+K_S^{i-j+1}\big|\_\cC(\Delta_{ij})\t^{i-j+1} \nonumber \\ &-K_S^{i-j}\big|\_\cC(\Delta_{ij})\t^{i-j}
-K_S^{i-j+2}\big|\_\cC(\Delta_{ij})\t^{i-j+1}\bigg], \label{Nvir3}
\end{align}
where $\Delta_{ij}$ is the divisor $\cZ_j-\cZ_i$ (effective if and only if $j\ge i$).

The moving part of \eqref{Nvir2} is (the K-theory class of) $N^{\vir}$, and will be used in Section \ref{virnorsec}. For now we concentrate on the fixed part --- i.e. the dual of the obstruction theory $\big[(E\udot)^f\big]^\vee$ of $P^T_X$. We also restrict to $C^{[\n]} \subset P^T_{(d)C}\subset P^T_X$, so $\cC$ becomes plain $C$. We set
$$
\Delta_i:=\Delta_{i-1,i}=\cZ_i-\cZ_{i-1} \quad\mathrm{of\ length\ }\delta_i:=n_i-n_{i-1},
$$
and use the standard isomorphism \cite{Che}
\begin{align} \label{isot}
C^{[\n]}\ &\Rt{\sim}\ C^{[n_0]} \times C^{[\delta_1]} \times \cdots \times C^{[\delta_{d-1}]}\,, \\ \nonumber
(Z_0, Z_1, \ldots, Z_{d-1})\ &\Mapsto\ (Z_0, \Delta_1,  \ldots, \Delta_{d-1}).
\end{align}
The fixed parts of \eqref{Nvir2} and \eqref{Nvir3} give $\big[(E\udot)^f\big]^\vee=
-\langle\II\udot,\II\udot\rangle^f_0$ as
$$
R\pi_{P*}\bigg[\O_C(\cZ_0)+\sum_{j=1}^{d-1}\Big(K_S\big|\_C
+\O_C(\Delta_j)-\O_C-K_S\big|_C(\Delta_j)\Big)+K_S\big|_C-\O_C\bigg]. $$
Simplifying gives
$$
\big[(E\udot)^f\big]^\vee\ =\ R\pi_{P*}\bigg[K_S\big|_C+\O_{\cZ_0}(\cZ_0)+
\sum_{i=1}^{d-1}\Big(\O_{\Delta_i}(\Delta_i)-K_S\big|_{\Delta_i}(\Delta_i)\Big)\bigg].
$$
The first term is the natural obstruction theory of $H_\k$, and the next two give the tangent bundle of $C^{[\n]}$ via the isomorphism \eqref{isot}. Subtracting the tangent terms leaves minus the K-theory class of the obstruction bundle, so
\beq{obbb}
\Big[\!\Ob\!\big|\_{C^{[\n]}}\Big]\ =\ \Big[R^1\pi_{P*}\big(K_S\big|_C\big)\Big]+
\sum_{i=1}^{d-1}\pi_{P*}\left[K_S\big|_{\Delta_i}(\Delta_i)\right].
\eeq

In particular the virtual dimension of $P^T_X$ at any point of $C^{[\n]}$ is $\chi(K_S|_C)+n_0$,
where $n_0$ is the length of $Z_0$. As $C$ is in the canonical class $\beta=\k$ we have $\chi(K_S|_C)=0$, so finally we obtain
\beq{virtdim2}
vd=n_0.
\eeq

We now substitute \eqref{obbb} into \eqref{530}. The first term of \eqref{obbb}
is the class of a trivial bundle over $P^T_{(d)C}$, so does not contribute. Therefore the cosection localised virtual cycle in $A_{vd}(C^{[\n]})$ is simply\footnote{As noted in Remark \ref{rmkvfc}, when $n_0=n_{d-1}$ this reduces to $(-1)^{\chi(\O_S)}\big[C^{[n_0]}\big]$.}
$$
(-1)^{\chi(\O_S)}\prod_{i=1}^{d-1}
c_{\mathrm{top}}\!\left(\!\pi_{P*}\big(K_S\big|_{\Delta_i\!}(\Delta_i)\big)\!\right)
\in A_{n\_0}\big(C^{[n\_0]} \times C^{[\delta_1]} \times \cdots \times C^{[\delta_{d-1}]}\big).
$$
This is easily calculated via relative Serre duality. Since $\Delta_i\subset C^{[\delta_i]}\times C$ is a divisor, its relative canonical bundle over $C^{[\delta_i]}$ is
$$
\omega_{\Delta_i/C^{[\delta_i]}}\ \cong\ \omega_C(\Delta_i)\big|_{\Delta_i}\ \cong\
K_S^2\otimes\O_{\Delta_i}(\Delta_i).
$$
Therefore the localised virtual cycle is
\begin{align*}
(-1)^{\chi(\O_S)}\prod_{i=1}^{d-1}c\_{\delta_i}\!\!\left(\!\pi_{P*}\big(K_S^{-1}\big|_{\Delta_i}\otimes
\omega_{\Delta_i/C^{[\delta_i]}}\big)\!\right)
&=(-1)^{\chi(\O_S)}\prod_{i=1}^{d-1}c\_{\delta_i}\!\left(\!\big(\pi_{P*}K_S\big|_{\Delta_i}\big)^{\!*\!}\right)\\
&=(-1)^{\chi(\O_S)}\prod_{i=1}^{d-1}(-1)^{\delta_i}c\_{\delta_i}\!\left(\!\big(K_S|\_C\big)^{\![\delta_i]}\right).
\end{align*}
Using the binomial convention \eqref{binomconv},
$$
\int_{C^{[k]}}c_k\big(L^{[k]}\big)={\deg L\choose k},
$$
for any line bundle $L$ on $C$. This is easiest to see when $L$ has a section $s$ with reduced zeros $z_1,\ldots,z_{\deg L}$. Then the induced section $s^{[k]}$ of $L^{[k]}$ has reduced zeros at precisely the points $(z_{i_1},\ldots,z_{i_k})$, where $\{i_1,\ldots,i_k\}$ is any subset of $\{1,\ldots,\deg L\}$.\footnote{More generally when $n \geq 2h-1$ Lemma VIII.2.5 of \cite{ACGH} gives an expression for $c_\mdot(L^{[k]})$. Combining with \eqref{Poin} below gives the formula. For general $n$ the formula follows using the ``embedding trick'' of Section \ref{universal}.} Putting it all together, we have proved the following.

\begin{proposition} \label{vircycfinal}
The Kiem-Li localised virtual cycle \eqref{530} is the multiple
$$
(-1)^{\chi(\O_S)+n_{d-1}-n_0}\prod_{i=1}^{d-1}{\k^2\choose\delta_i}
$$
of the cycle
\beq{cycle}
\big[C^{[n_0]}\big]\times[\pt]\times\ldots\times[\pt]
\eeq
in $A_{n_0}\big(C^{[n_0]}\times C^{[\delta_1]}\times\ldots\times C^{[\delta_{d-1}]}\big)
=A_{n_0}(C^{[\n]}).\hfill\square$
\end{proposition}

\section{The virtual normal bundle} \label{virnorsec}

We want to calculate the contribution of the vertical component $[P_{\ver}^{T}]^{\vir}$ \eqref{KLvir} to the invariants \eqref{generalinv}. By Proposition \ref{vircycfinal} we can now pull everything back to an integral on $C^{[n_0]}$. We do this first with the virtual normal bundle. We use the projections
\beq{projns}
\xymatrix@=15pt{
& C^{[\n]}\times C \ar[dl]_\pi \ar[dr]^{p\_C} & \\
C^{[\n]} & & C}
\eeq
and usually suppress $p_C^*$ as before. We also use the standard notation
$$
c_s(E):=1+s \, c_1(E)+s^2 \, c_2(E)+\ldots
$$
for any complex of sheaves $E$. When $E$ is a vector bundle of rank $r$ we have
\begin{equation} \label{param}
e(E \otimes \t^w) = \sum_{i=0}^{r} c_i(E) (w t)^{r-i} = (w t)^r c_{1/wt}(E) = (w t)^r c_{-1/wt}(E^\vee).
\end{equation}
Therefore the same identity holds for $E=\{\cdots\to E^i\to E^{i+1}\to\cdots\}$ a finite complex of rank $r:=\sum_i(-1)^i\rk(E^i)$.
In particular, when $E$ is a trivial bundle (or constant complex) we have
\begin{equation} \label{free}
e(E \otimes \t^w) = (w t)^r.
\end{equation}
%

\begin{proposition} \label{Nvirforcalc}
The pull-back of $\frac{1}{e(N^{\vir})}$ to the cycle $C^{[n_0]}$ of \eqref{cycle}
equals
$$
A\,t^{n_0\,}c_{-1/dt}(E)\prod_{i=1}^{d-1} c_{-1/it}(F_i)
$$
where
\begin{eqnarray} 
\begin{split} \label{EFA}
E &= R\pi_*\Big[K_{S}|_{C}^{-(d-1)}(\cZ_0 + \Delta_{0,d-1})\Big], \\
F_i &= R\pi_* \Big[K_{S}|_{C}^{-(i-1)} (\cZ_0 + \Delta_{0,i-1}) \otimes (\O_C - K_{S}|_{C}^{-1}(\Delta_i)) \Big],\ \ \mathrm{and} \\ 
A &= (-1)^{\frac{1}{2}d(d-1)\k^2 + \sum_{i=1}^{d-1} n_i\ } \bigg( \frac{d!}{d^d} \bigg)^{\!\!\k^2} d^{\,n_{d-1}} \prod_{i=1}^{d-1} i^{\,-\delta_i}\,. 
\end{split}
\end{eqnarray}
\end{proposition}
\begin{proof}
Taking the moving parts of \eqref{Nvir2} and \eqref{Nvir3} we find
\begin{align} 
&\frac{1}{e(N^{\vir})} = e((R \pi_* \O(\cZ_0))^{\vee} \otimes \mathfrak{t}) \prod_{i=1}^{d-1} \frac{e((R \pi_* K_{S}|_{C}^{-i}(\cZ_i))^{\vee} \otimes \mathfrak{t}^{i+1})}{e((R \pi_* K_{S}|_{C}^{-i}(\cZ_i)) \otimes \mathfrak{t}^{-i})} \label{Nvirstart} \\
\times \prod_{{\scriptsize{\begin{array}{c} i,j = 0 \\ i \neq j \end{array}}}}^{d-1} &\frac{e(R \pi_* K_{S}|_{C}^{i-j}(\Delta_{ij}) \otimes \mathfrak{t}^{i-j})}{e(R \pi_* K_{S}|_{C}^{i-j+1}(\Delta_{ij}) \otimes \mathfrak{t}^{i-j})} \prod_{{\scriptsize{\begin{array}{c} i,j = 0 \\ i+1 \neq j \end{array}}}}^{d-1} \frac{e(R \pi_* K_{S}|_{C}^{i-j+2}(\Delta_{ij}) \otimes \mathfrak{t}^{i-j+1})}{e(R \pi_* K_{S}|_{C}^{i-j+1}(\Delta_{ij}) \otimes \mathfrak{t}^{i-j+1})}\,. \nonumber
\end{align}
We start with the second line of \eqref{Nvirstart}. On our cycle
$C^{[n_0]} \times \Delta_1 \times \ldots \times \Delta_{d-1}$ the divisors
$$
\Delta_{ij} = \left\{\!\!\begin{array}{cc} \sum_{k=i+1}^{j} \Delta_k, &  \ j>i \\ \sum_{k=j+1}^{i} \Delta_k, & \ i>j \end{array} \right.
$$
are fixed, since the divisors $\Delta_k$ are. Therefore each
$R \pi_* K_{S}|_{C}^{i-j+l}(\Delta_{ij})$ is a constant complex $\O_{C^{[n_0]}}^{\oplus r}$, where
$$
r=\chi\big(K_{S}|_{C}^{i-j+l}(\Delta_{ij})\big)=n_j-n_i + (i-j+l-1)\k^2
$$
by Riemann-Roch.

So by \eqref{free} the second line of \eqref{Nvirstart} is
$$
\prod_{{\scriptsize{\begin{array}{c} i,j = 0 \\ i \neq j \end{array}}}}^{d-1} \frac{((i-j)t)^{n_j-n_i+(i-j-1)\k^2}}{((i-j)t)^{n_j-n_i+(i-j)\k^2}} \prod_{{\scriptsize{\begin{array}{c} i,j = 0 \\ i+1 \neq j \end{array}}}}^{d-1} \frac{((i-j+1)t)^{n_j-n_i+(i-j+1)\k^2}}{((i-j+1)t)^{n_j-n_i+(i-j)\k^2}} $$
which simplifies to
$$
\prod_{{\scriptsize{\begin{array}{c} i,j = 0 \\ i \neq j \end{array}}}}^{d-1} \frac{1}{((i-j)t)^{\k^2}} \prod_{{\scriptsize{\begin{array}{c} i,j = 0 \\ i+1 \neq j \end{array}}}}^{d-1} ((i+1-j)t)^{\k^2}.
$$
The only terms in this expression which do not cancel immediately are those with $i=0$ in the first product and $i=d-1$ in the second product. This gives
\begin{equation} \label{cst1}
\prod_{j=1}^{d-1} \frac{1}{(-jt)^{\k^2}} \cdot \prod_{j=0}^{d-1} ((d-j)t)^{\k^2} = (-1)^{(d-1)\k^2} d^{\,\k^2} t^{\k^2}.
\end{equation}

We now deal with the first line of \eqref{Nvirstart}. Applying \eqref{param} gives 
\beq{ecur}
t^{n_0 - \k^2} c_{-1/t}(R \pi_* \O(\cZ_0))\,\prod_{i=1}^{d-1} \frac{((i+1)t)^{n_i - (i+1)\k^2}}{(-it)^{n_i - (i+1)\k^2}}\ \prod_{i=1}^{d-1} \frac{c_{-1/(i+1)t}(R \pi_* K_{S}|_{C}^{-i}(\cZ_i))}{c_{-1/it}(R \pi_* K_{S}|_{C}^{-i}(\cZ_i))}\,.
\eeq
The first product can be simplified as
\begin{align}
\begin{split} \label{cst2}
&(-1)^{\sum_{i=1}^{d-1}(n_i - (i+1)\k^2)} \prod_{i=1}^{d-1} \Big( \frac{i+1}i \Big)^{-(i+1) \k^2}\ \prod_{i=1}^{d-1} \Big( \frac{i+1}{i} \Big)^{n_i} \\
=\ & (-1)^{\big(\frac{1}{2}d(d+1)-1\big)\k^2 + \sum_{i=1}^{d-1} n_i} \prod_{i=1}^{d-1} \Big( \frac{i}{i+1} \Big)^{(i+1) \k^2}\left(\prod_{i=1}^{d-1} i^{\,n_{i-1} - n_i}\!\right)d^{\,n_{d-1}} \\
=\ & (-1)^{\frac{1}{2}d(d-1)\k^2 + (d-1)\k^2 + \sum_{i=1}^{d-1} n_i} \left( \frac{(d-1)!}{d^d} \right)^{\!\!\k^2} d^{\,n_{d-1}}\prod_{i=1}^{d-1} i^{\,-\delta_i}.
\end{split}
\end{align}
Multiplying $t^{n_0-\k^2}$, \eqref{cst2} and \eqref{cst1} together gives $A\,t^{n_0}$, as required. 

What remains in \eqref{ecur} is  
$$
c_{-1/t}(R \pi_* \O(\cZ_0)) \prod_{i=1}^{d-1} \frac{c_{-1/(i+1)t}(R \pi_* K_{S}|_{C}^{-i}(\cZ_0+\Delta_{0,i}))}{c_{-1/it}(R \pi_* K_{S}|_{C}^{-i}(\cZ_0+\Delta_{0,i}))}\,.
$$
Reordering the product gives
$$
\prod_{i=1}^{d-1}\frac{c_{-1/it}(R \pi_* K_{S}|_{C}^{-(i-1)}(\cZ_0+\Delta_{0,i-1}))}{c_{-1/it}(R \pi_* K_{S}|_{C}^{-i}(\cZ_0+\Delta_{0,i}))}\cdot c_{-1/dt}(R \pi_* K_{S}|_{C}^{-(d-1)}(\cZ_0+\Delta_{0,d-1})),
$$
which we write as
$$
\left[\prod_{i=1}^{d-1}
c_{-1/it}\big(R \pi_*\big(K_{S}|_{C}^{-(i-1)}(\cZ_0+\Delta_{0,i-1})-K_{S}|_{C}^{-i}(\cZ_0+\Delta_{0,i})\big)\big)\right]
c_{-1/dt}(E).
$$
This is $c_{-1/dt}(E)\prod_{i=1}^{d-1} c_{-1/it}(F_i)$ as claimed.
\end{proof}

\section{Descendent insertions} \label{descendsec}

Recall from Section \ref{sectionfixedlocus} that given a cohomology class $\sigma\in H^*(S,\Q)$ and a nonnegative integer $\alpha$, we defined in \eqref{defdesc} the descendent insertion
$$
\tau_\alpha(\sigma)\,\in\,H^*_{T}(P_X,\Q).
$$
We have localised the vertical component $[P_{\ver}^T]^{\vir}$ of the virtual cycle $[P_{X}^{T}]^{\vir}$ to
$$
P_{(d)C}^T=\bigsqcup_{\n}C^{[\n]}.
$$
We next restrict the descendents to $C^{[\n]}$. We use the projections \eqref{projns} and the universal divisors $\cZ_0\subset\cdots\subset\cZ_{d-1}\subset C^{[\n]}\times C$.

\begin{proposition} \label{E(x)}
Let $E(x):=\frac{1-e^{-x}}x=1-x/2!+x^2/3!-\ldots\ $. The restriction of $\tau_\alpha(\sigma)$ to $C^{[\n]}\subset P_X$ is the degree\footnote{This is the \emph{real} cohomological degree; twice the complex degree.} $2\alpha+\deg\sigma-2$ part of
$$
\pi_{*\!}\left[p_C^*\left(\sigma|\_C-\frac{\sigma.\k}2\Big|_C\right)E(\k+t)\,\sum_{j=0}^{d-1}
e^{[\cZ_j]-j(\k+t)}\right].
$$
\end{proposition}

\begin{proof}
As in \eqref{decom}, the K-theory class of the restriction of the universal sheaf $\FF$ to $C^{[\n]}\times X$ is
$$
\big[\FF\big]=\sum_{j=0}^{d-1}\big[\O_{C^{[\n]}\times C}(\cZ_j)\otimes K_S^{-j}\otimes\t^{-j}\big].
$$
Let $i$ denote both the inclusion $C\into X$ and its basechange $C^{[\n]}\times C\into C^{[\n]}\times X$. It has normal bundle $\nu_C=K_S|_C\,\oplus\,K_S|_C\otimes\t$, so by $T$-equivariant Grothendieck-Riemann-Roch \cite{EG},
\begin{eqnarray}
\ch^T(\FF) &=& \sum_j\ch(i_*\O(\cZ_j))e^{-j(\k+t)} \nonumber \\
&=& \sum_j i_*\left(\ch(\O(\cZ_j))\td^{-1}(\nu_C)\right)e^{-j(\k+t)} \nonumber \\
&=& i_*\sum_je^{[\cZ_j]}\,E(\k)\,E(\k+t)\,e^{-j(\k+t)}. \label{abf}
\end{eqnarray}
If we write this as $i_*A$ then, again on restriction to $C^{[\n]}\subset P_X$ we find
$$
\tau_\alpha(\sigma)=
\pi_{P*} \big( \pi_X^{*} \sigma \cap[i_*A]_{2\alpha+4}\big)=
\pi_* \big(i^*\pi_X^{*} \sigma\cap[A]_{2\alpha}\big)
$$
because $\pi=\pi_P\circ i$. Recalling the identification \eqref{HtoHT}, we also see that $i^*\pi_X^{*} \sigma=p_C^*(\sigma|_C)$. Substituting into \eqref{abf} gives
$$
\pi_*\!\left[p_C^*(\sigma|_C)E(\k)E(\k+t)\sum_{j=0}^{d-1}
e^{[\cZ_j]-j(\k+t)}\right]_{2\alpha+\deg\sigma}\ ,
$$
which simplifies to the required formula.
\end{proof}

\begin{corollary} \label{EB}
Let $D \in H_2(S)$. Then on restriction to the cycle $C^{[n_0]}$ of \eqref{cycle} we find that $\tau_\alpha(\sigma)$ is the degree $2\alpha$ part of
$$
(\k\cdot\! D)E(t)\sum_{j=0}^{d-1}
e^{\omega-jt},
$$
where $\omega$ is the class of the divisor $\cZ_0\subset C^{[n_0]}\times C$ restricted to $C^{[n_0]}\times\{c_0\}$ (and $c_0\in C$ is any basepoint).
\end{corollary}

In the formula of Theorem \ref{main} we only consider insertions $\tau_{\alpha_j}(D_j)$ coming from $D_j\in H_2(S)$. Expanding as a polynomial in $\omega$,
\beq{gamma}
\prod_{j=1}^{m} \tau_{\alpha\_j}(D_j)\ =\ \prod_{i=1}^{m} (\k \cdot\!D_i) \left[ E(t) \sum_{j=0}^{d-1}e^{\omega- j t}\right]_{2\alpha_i}\!=\ \prod_{j=1}^{m} (\k \cdot\! D_j)\sum_{a=0}^{\infty} \gamma_{a\,}\omega^a,
\eeq
for some $\gamma_a\in\Q[t]$ whose precise form we do not need. Here $[\ \cdot\ ]_{2\alpha_i}$ denotes the degree $2\alpha_i$ part in the degree 2 variables $\omega$ and $t$.

Therefore by Propositions \ref{vircycfinal} and \ref{Nvirforcalc}, $\sfZ_{d\k}^P(X,\tau_{\alpha\_1}(D_1) \cdots \tau_{\alpha\_m}(D_m))_{\ver}$ equals
\beq{integrand}
\prod_{j=1}^{m} (\k \cdot\! D_j)
\sum_{0 \leq n_0 \leq \cdots \leq n_{d-1};\ a \geq 0} q^\chi \, t^{n_0} \, \gamma_a \, B \int_{C^{[n_0]}}\omega^ac_{-1/dt}(E)
\prod_{i=1}^{d-1} c_{-1/it}(F_i),
\eeq
where $E,F_i$ are defined in \eqref{EFA},
$\chi={\sum_{i=0}^{d-1}(n_i - (i+1) \k^2)}$ by \eqref{neqn2}, and
\beqa
B&=&(-1)^{\chi(\O_S)+n_{d-1}-n_0}A\,\prod_{i=1}^{d-1}{\k^2\choose\delta_i} \\
&=&(-1)^{\frac{1}{2}d(d-1)\k^2 +\chi(\O_S)+ \sum_{i=0}^{d-1} n_i\ } \bigg( \frac{d!}{d^d} \bigg)^{\!\!\k^2}(-d)^{n_{d-1}} \prod_{i=1}^{d-1}\left[i^{\,-\delta_i}{\k^2\choose\delta_i}\!\right].
\eeqa

\section{Expression in terms of tautological classes} \label{step2}

We now write the integrand of \eqref{integrand} in terms of tautological classes on the symmetric product $C^{[n_0]}$. For now we assume, for simplicity, that
$$
n_0 > 2h -2,
$$
where $h=\k^2+1$ is the canonical genus; later we will explain how to remove this assumption. Therefore the Abel-Jacobi map
\begin{align*}
\AJ\colon\ C^{[n_0]} &\ \longrightarrow\ \Pic^{n_0}(C), \\
Z_0 &\ \Mapsto\ \O_C(Z_0),
\end{align*}
is a projective bundle. In fact, using the notation
$$\xymatrix@C=30pt{
C^{[n_0]}\times C \ar[r]^-{\AJ\times1\,}\ar[d]^{\pi_1} &  \Pic^{n_0}(C) \times C \ar[d]^{\pi_2} \\ C^{[n_0]} \ar[r]^-{\AJ}& \Pic^{n_0}(C)
}$$
and letting $\cP$ be a Poincar\'e line bundle on $\Pic^{n_0}(C) \times C$, we have
$$C^{[n_0]}\ =\ \PP(\pi_{2*} \cP).$$
We normalise $\cP$ by fixing
\beq{normali}
\cP\big|_{\Pic^{n_0}(C)\times\{c_0\}}\ \cong\ \O_{\Pic^{n_0}(C)\,},
\eeq
by tensoring it with $\pi_2^*\big(\cP^{-1}|_{\Pic^{n_0}(C)\times\{c_0\}}\big)$ if necessary.
This fixes a tautological line bundle
\beq{tautin}
\O(-1)\ \subset\ \AJ^* \pi_{2*} \cP
\eeq
on $C^{[n_0]}$, and so the tautological class
\beq{omdef}
\omega\, :=\ c_1(\O(1))\,\in\,H^2(C^{[n_0]},\Z).
\eeq
The map $\pi_2^*\pi_{2*} \cP\to\cP$, pulled back along $\AJ\times1$ and composed with the inclusion \eqref{tautin}, gives a canonical section of $(\AJ\times1)^*\cP(1)$
vanishing on the universal divisor $\cZ_0\subset C^{[n_0]}\times C$. Therefore
\beq{lineq}
(\AJ\times1)^*\cP(1)\,\cong\,\O(\cZ_0) \quad\mathrm{and\ so}\quad
\O(1)\,\cong\,\O(\cZ_0)\big|_{C^{[n_0]}\times\{c_0\}}
\eeq
by the normalisation condition \eqref{normali}. In particular the $\omega$ of \eqref{omdef} is the divisor class $\big[\cZ_0|_{C^{[n_0]}\times\{c_0\}}\big]$, and so is the same $\omega$ as appears in Corollary \ref{EB}.

The second tautological class we use is the pullback of the class of the theta divisor on $\Pic^{n_0}(C)$,
$$
\theta \in H^2(\Pic^{n_0}(C),\Z) \,\cong\,\Hom(\Lambda^2 H^1(C,\Z),\Z)
$$
which takes $\alpha,\beta\in H^1(C,\Z)$ to $\int_C \alpha \wedge \beta$.
We denote its pullback $\AJ^* \theta$ to $C^{[n_0]}$ by $\theta$ also.
 
\begin{proposition} \label{taut}
The integrand $\omega^ac_{-1/dt}(E)
\prod_{i=1}^{d-1} c_{-1/it}(F_i)$ in \eqref{integrand} can be written in terms of the tautological classes $\omega,\,\theta$ as
\beq{grand}
\omega^a\sum_{k=0}^{\infty} \left(1-\frac{\omega}{dt}\right)^{n_{d-1} - d \k^2 - k\ } \frac{(\theta/dt)^k}{k!}\ 
\prod_{i=1}^{d-1} \left(1 - \frac{\omega}{it} \right)^{\!\k^2-\delta_i}.
\eeq
\end{proposition}

\begin{proof}
By \eqref{lineq} we see the complex $E$ \eqref{EFA} satisfies
\beq{complexx}
E(-1)\ =\ R\pi_{1*}\big((\AJ\times1)^*\cP\otimes K_{S}|_{C}^{-(d-1)}(\Delta_{0,d-1})\big),
\eeq
where we recall that we work with fixed $\Delta_{0,d-1} = \Delta_0 + \cdots + \Delta_{d-1}$. We begin by computing the Chern \emph{character} of this. By Grothendieck-Riemann-Roch,
\begin{eqnarray*}
\ch(E(-1)) &=& \pi_{1*}\big[\ch \big((\AJ\times1)^*\cP\otimes K_{S}|_{C}^{-(d-1)}(\Delta_{0,d-1})\big)\td(C)\big] \\ &=& \pi_{1*} \Big[\exp\!\big((\AJ\times1)^*c_1(\cP) -(d-1)\k+[\Delta_{0,d-1}]\big) (1+c_1(T_C)/2) \Big] \\
&=& n_0 - (d-1) \k^2 + (n_{d-1}-n_0) -\k^2-\theta \\
&=& n_{d-1} - d\,\k^2 - \theta,
\end{eqnarray*}
where we have identified $c_1(\cP)$ with
\beqa
\big(0,\id,n_0[c_0]\big) &\in&
H^2(\Pic^{n_0}(C))\ \oplus\ \big(H^1(C)^* \otimes H^1(C)\big)\ \oplus\ H^2(C) \\ &=& H^2\big(\Pic^{n_0}(C) \times C\big)
\eeqa
using the normalisation condition \eqref{normali}. We also used the (pullback by $\AJ\times1$ of the) standard identity \cite[Section VIII.2]{ACGH}
$$
\frac{1}{2}\pi_{2*}\big(\id^{\wedge2}\big)\,=\,-\theta.
$$
Therefore \eqref{complexx} has rank $n_{d-1}-d\,\k^2$, first Chern class $-\theta$, and all higher Chern \emph{characters} vanish. From this we deduce that
\beq{cherns}
c_k\ =\ \frac{(-\theta)^k}{k!} \quad \mathrm{for \ all \ } k>0.
\eeq
So now applying the identity
\beq{expand}
c_s(V(1))\ =\ \sum_{k=0}^{\infty} (1+\omega s)^{\rk(V)-k} c_k(V) s^k
\eeq
to \eqref{cherns} we obtain
$$
c_{-1/dt}(E)\ =\ 
\sum_{k=0}^{\infty} \left(1-\frac{\omega}{dt}\right)^{n_{d-1} - d \k^2 - k} \frac{(-\theta)^k}{k!} \left(\frac{-1}{dt}\right)^{\!\!k}.
$$
This gives the first term of the integrand. The second is easier. By \eqref{lineq} again,
$$
F_i(-1)\ =\ R \pi_{1*} \big[ (\AJ\times1)^*\cP \otimes K_{S}|_{C}^{-(i-1)} (\Delta_{0,i-1}) \otimes \big\{ \O_C - K_{S}|_{C}^{-1}(\Delta_i) \big\}\big].
$$
By Proposition \ref{vircycfinal}, each $\Delta_i \subset C^{[n_0]} \times C$ pulls back from $C$. In the (numerical) K-group we can write
\begin{align*}
\O_C - K_{S}|_{C}^{-1}(\Delta_i) &= \O_C - K_{S}|_{C}^{-1} - \delta_i \cdot \O_{c} \\
&= (\k^2 - \delta_i) \cdot \O_{c},
\end{align*}
where $c \in C$ is any point and $\k^2 = \deg K_S|_C$. Therefore, by the normalisation condition \eqref{normali}, $F_i(-1)$ equals $(\k^2-\delta_i)\cdot \O_{C^{[n_0]}}$ in the K-group. 
By \eqref{expand} we find
\[
c_{-1/it}(F_i) = \left(1-\frac{\omega}{it} \right)^{\!\k^2-\delta_i}. \qedhere
\]
\end{proof}

\section{Evaluation of the integral} \label{step3}

Still working under the assumption $n_0 > 2h-2$ for the time being, we can now compute the integral in \eqref{integrand}.

\begin{proposition} \label{answerint}
The integral of \eqref{grand} over $C^{[n_0]}$ is
$$
\sum(dt)^{-n_0+a} \binom{n_0 - n_{d-1} + (d+1)\k^2 - a - |{\bf j}|}{n_0 - a - |{\bf j}|} \prod_{i=1}^{d-1} \left(\frac{-d}{i}\right)^{\!j_i\!\!} \binom{\k^2-\delta_i}{j_i},
$$ 
where the sum is over all $j_1, \ldots, j_{d-1} \geq 0$, and we set $|{\bf j}|:=j_1 + \cdots + j_{d-1}$.
\end{proposition}

\begin{proof}
Expanding \eqref{grand} by the binomial theorem using the convention \eqref{binomconv} gives the sum over all $k,l,j_1, \ldots, j_{d-1} \geq 0$ of
$$
\left[\Big(\frac{1}{dt}\Big)^k \Big(\frac{-1}{dt}\Big)^l \binom{n_{d-1} - d\,\k^2 - k}{l} \prod_{i=1}^{d-1} \Big(\frac{-1}{it}\Big)^{j_i} \binom{\k^2-\delta_i}{j_i} \right] \frac{\theta^k}{k!}\,\omega^{a+l+|{\bf j}|}.
$$
We can now integrate over $C^{[n_0]}$ using \cite[Section VIII.3]{ACGH}:
\beq{Poin}
\int_{C^{[n_0]}} \frac{\theta^k}{k!}\, \omega^{n_0-k}\ =\ \binom{h}{k}, \ \ \textrm{for all } k \in [0,n_0],
\eeq
where $h=\k^2+1$ is the genus of $C$. This gives the sum over all $j_1, \ldots, j_{d-1} \geq 0$ and $k\in[0,n_0]$ of
$$
\Big(\frac{1}{dt}\Big)^k \Big(\frac{-1}{dt}\Big)^{n_0-a-k-|{\bf j}|}
{h\choose k}\binom{n_{d-1} - d\,\k^2 - k}{n_0 - a - k - |{\bf j}|} \prod_{i=1}^{d-1}
\Big(\frac{-1}{it}\Big)^{j_i} \binom{\k^2-\delta_i}{j_i}.
$$
We can sum over all $k\ge0$ since ${h\choose k}=0$ for $k>n_0\ge2h-1\ge h$ when $h \geq 1$ (and when $h=0$ it is also clear we can sum over all $k \geq 0$). So using
$$
\binom{a}{b} \,=\, (-1)^b \binom{b-a-1}{b}
$$
we get the sum over all $k,j_1, \ldots, j_{d-1} \geq 0$ of
$$
(dt)^{-n_0+a+|{\bf j}|}\binom{h}{k}\!\binom{n_0-a-|{\bf j}| - n_{d-1}+d\k^2 -1}{n_0 - a - k - |\bf j|} \prod_{i=1}^{d-1}\left(\!\frac{-1}{it}\right)^{\!j_i\!} \binom{\k^2-\delta_i}{j_i}.
$$
Summing over $k$ using the Chu-Vandermonde identity
$$
\sum_{k = 0}^{\infty} \binom{a}{c-k} \binom{b}{k} = \binom{a+b}{c}
$$
gives the claimed formula.
\end{proof}

\subsection{Extension to all $n_0$} \label{universal}
We established Proposition \ref{answerint} assuming $n_0 > 2h-2$. However the answer holds for \emph{any} $n_0$. For general $n_0$, pick $N > n_0$ such that $N > 2h-2$. Then $C^{[N]} \cong \PP(\pi_{2*} \cQ)$, where $\cQ$ is the normalised Poincar\'e bundle on $\Pic^N(C)\times C$. We can embed\footnote{The method described here was used in the case of the Hilbert scheme of curves on surfaces in \cite{DKO} and also \cite{KT2}.}
\beqa
C^{[n_0]} &\Into& C^{[N]}, \\
Z_0 &\Mapsto& Z_0+(N-n_0)c_0.
\eeqa
Denote the universal divisor on $C^{[N]} \times C$ by $\cW$, and let $s\in H^0(\O(\cW))$ be the section cutting it out. Then $C^{[n_0]}\subset C^{[N]}$ is the locus of effective divisors containing $(N-n_0)c_0$; i.e.~it is the locus where $s$ vanishes on restriction to the Artinian thickened point $(N-n_0)c_0$. Denote the restriction to $C^{[N]} \times (N-n_0)c_0$ of $\pi_2 : C^{[N]} \times C \rightarrow C^{[N]}$ by $\pi_2$ as well. Then $\pi_{2*}(s|_{C^{[N]} \times(N-n_0)c_0})$ defines a section of the locally free sheaf
\[
F := \pi_{2*}\big(\O(\cW)|_{C^{[N]} \times (N-n_0)c_0}\big)
\]
which cuts out $C^{[n_0]}$. The rank of $F$ is the codimension of $C^{[n_0]}$, so it is a regular section and we can identify the normal bundle
$$N_{C^{[n_0]} / C^{[N]}} \,\cong\, F|_{C^{[n_0]}}$$
and the cycle class
\beq{cla}
\big[C^{[n_0]}\big]\ =\ \big[c_{N-n_0}(F)\big]\in H_{2(N-n_0)}\big(C^{[N]}\big).
\eeq
All results of Sections \ref{step2} and \ref{step3} can be obtained by pushing forward to $C^{[N]}$ and then pushing down $\AJ$ using the commutative diagram
$$\xymatrix@C=65pt{
C^{[n_0]} \ar[d]^{\AJ}\INTO^{+(N-n_0)c_0} & C^{[N]} \ar[d]^{\AJ} \\
\Pic^{n_0}(C) \ar[r]^{\otimes\O((N-n_0)c_0)\ }_\simeq& \Pic^N(C).}
$$
Pushing forward to $C^{[N]}$ introduces the class \eqref{cla}, while
$\cP$ gets replaced by $\cQ(-(N-n_0)c_0)$ and $\cZ_0$ gets replaced by $\cW - (N-n_0)c_0$. The calculation proceeds in exactly the same manner except for one difference: the usual relation $\AJ_*\omega^{i+n_0-h}=\theta^i/i!$ that goes into the Poincar\'e formula \eqref{Poin} for $n_0>2h-2$ is replaced by the identity
\begin{equation*}
\AJ_*\big(c_{N-n_0}(F) \omega^{i+n_0-h}\big)\ =\ \left\{\!\!\begin{array}{cl} \frac{\theta^i}{i!} & \mathrm{if} \ i \geq 0, \\ 0 & \mathrm{otherwise}; \end{array} \right.
\end{equation*}
see for instance \cite[Lemma 4.3]{KT2}.\footnote{Although \cite[Lemma 4.3]{KT2} is derived for the Hilbert scheme of curves on a surface the same formula holds in the (easier) setting of the Hilbert scheme of points on a curve.} This removes the extra class \eqref{cla} and produces the same formulae as for $n_0>2h-2$.

\section{Final formula without descendents} \label{step4}

Plugging Proposition \ref{answerint} into \eqref{integrand} evaluates
$\sfZ_{d\k}(X,\tau_{\alpha_1}(D_1) \cdots \tau_{\alpha_m}(D_m))_{\ver}$ as the sum over all $a,\,j_1, \ldots, j_{d-1} \geq 0$ and all $0 \leq n_0 \leq \cdots \leq n_{d-1}$ of
\begin{multline*}
(-1)^{\chi(\O_S)+\frac{1}{2}d(d-1)\k^2 + \sum_{i=0}^{d-1} n_i\ } \bigg( \frac{d!}{d^d} \bigg)^{\!\!\k^2}(-d)^{n_{d-1}} \prod_{i=1}^{d-1}\left[i^{\,-\delta_i}{\k^2\choose\delta_i}\!\right]
\,t^{n_0}\prod_{j=1}^{m} (\k \cdot\! D_j) \\
\times q^{\chi\,}\gamma_a
(dt)^{-n_0+a} \binom{n_0 - n_{d-1} + (d+1)\k^2 - a - |{\bf j}|}{n_0 - a - |{\bf j}|} \prod_{i=1}^{d-1} \left(\frac{-d}{i}\right)^{\!j_i\!\!} \binom{\k^2-\delta_i}{j_i}.
\end{multline*}
Here the exponent of $q$ is
$$
\chi\ =\ \sum_{i=0}^{d-1}\big(n_i - (i+1) \k^2\big)\ =\ dn_0+\sum_{i=1}^{d-1}(d-i)\delta_i-\frac12d(d+1)\k^2.
$$
We combine the first and third products, collect powers of $d$ and $t$, and write each $n_i$ as $n_0+\delta_1+\ldots+\delta_i$. The result is the sum over $a,n_0\ge0$ and all $j_i,\delta_i\ge0$ of
\begin{multline*}
(-1)^{\chi(\O_S)+\frac{1}{2}d(d-1)\k^2}\bigg( \frac{d!}{d^d} \bigg)^{\!\!\k^2}
\left[\prod_{i=1}^{d-1}{\k^2\choose\delta_i}\binom{\k^2-\delta_i}{j_i}
\left(\frac{-d}i\right)^{\!\delta_i+j_i}(-q)^{(d-i) \delta_i}\!\right]
\\ \times q^{-\frac{1}{2}d(d+1)\k^2}\left[\prod_{j=1}^m(\k\cdot\!D_j)\right]\!\gamma_a
(dt)^a \binom{(d+1)\k^2- |{\bf \delta}| - a - |{\bf j}|}{n_0 - a - |{\bf j}|}\big(\!-(-q)^d\big)^{n_0},
\end{multline*}
where we have used $|\delta|$ to denote $\delta_1+\ldots+\delta_{d-1}=n_{d-1}-n_0$.

Remarkably this horrible-looking expression can be summed. The sum over $n_0$ only involves the last 2 terms; using our convention \eqref{binomconv} it takes the form
$$
C\sum_{n_0\ge0}\binom{r}{n_0-s}x^{n_0}\ =\ Cx^s(1+x)^r.
$$
This replaces the last two terms with
$$
\big(\!-(-q)^d\big)^{\!a+|{\bf j}|}\big(1-(-q)^d\big)^{(d+1)\k^2- |{\bf \delta}| - a - |{\bf j}|}.
$$
Setting $Q:=-q$, we write this as
$$
\big(\!-Q^d\big)^{\!a}(1-Q^d)^{2\k^2 - a}\prod_{i=1}^{d-1}(-Q^d)^{j_i}(1-Q^d)^{\k^2-\delta_i-j_i}.
$$
Combining with the $\binom{\k^2-\delta_i}{j_i}\left(\frac{-d}i\right)^{\!j_i}$ term we can now sum over $j_i\ge0$ using the binomial theorem again to give
\begin{multline*}
(-1)^{\chi(\O_S)+\frac{1}{2}d(d-1)\k^2}\bigg( \frac{d!}{d^d} \bigg)^{\!\!\k^2}
\left[\prod_{i=1}^{d-1}{\k^2\choose\delta_i}\left(\frac{-d}i\right)^{\!\delta_i}
\left((1-Q^d)+\frac{dQ^d}i\right)^{\!\!\k^2-\delta_i}\!\!Q^{(d-i) \delta_i}\!\right]
\\ \times (-Q)^{-\frac{1}{2}d(d+1)\k^2}\left[\prod_{j=1}^m(\k\cdot\!D_j)\right]\!\gamma_a
(dt)^a\big(\!-Q^d\big)^{\!a}(1-Q^d)^{2\k^2 - a}.
\end{multline*}
Moving $\big((d-1)!\big)^{\k^2}=\prod_{i=1}^{d-1}i^{\,\k^2}$
inside the product gives
\begin{multline*}
(-1)^{\chi(\O_S)+d\,\k^2}\bigg( \frac d{d^d} \bigg)^{\!\!\k^2}
\left[\prod_{i=1}^{d-1}{\k^2\choose\delta_i}
\big(i(1-Q^d)+dQ^d\big)^{\!\k^2-\delta_i}(-dQ^{(d-i)})^{\delta_i}\!\right]
\\ \times Q^{-\frac{1}{2}d(d+1)\k^2}\left[\prod_{j=1}^m(\k\cdot\! D_j)\right]\!\gamma_a
(dt)^a\big(\!-Q^d\big)^{\!a}(1-Q^d)^{2\k^2 - a}.
\end{multline*}
So now we can sum over all $\delta_i\ge0$ (by the binomial theorem again) and $a\ge0$ to give the full expression:
\begin{multline*}
(-1)^{\chi(\O_S)+d\,\k^2}\bigg( \frac1d\bigg)^{\!\!(d-1)\k^2}
\left[\prod_{i=1}^{d-1}\big(i(1-Q^d)+dQ^d-dQ^{d-i}\big)^{\k^2}\!\right]
\\ \times\left[\prod_{i=1}^{d-1}Q^{-\frac12d\k^2}\right]Q^{-d\,\k^2}(1-Q^d)^{2\k^2}
\left[\prod_{j=1}^m(\k\cdot\! D_j)\right]\!\sum_{a\ge0}\gamma_a
(dt)^a\left(\!\frac{-Q^d}{1-Q^d}\right)^{\!\!a}.
\end{multline*}
Combining the first two products gives
\begin{align} \nonumber
(-1)^{\chi(\O_S)+d\,\k^2}\bigg( \frac1d \bigg)^{\!\!(d-1)\k^2}\!\!\big(Q^{-d/2}-Q^{d/2}\big)^{\!2\k^2}\
&\prod_{i=1}^{d-1}\Big((d-i)Q^{d/2}-dQ^{d/2-i}+iQ^{-d/2}\Big)^{\!\k^2}
\\ &\times\left[\prod_{j=1}^m(\k\cdot\! D_j)\right]\!\sum_{a\ge0}\gamma_a
\left(\!\frac{dtQ^d}{Q^d-1}\right)^{\!a}. \label{atlast}
\end{align}
When there are no insertions the second line is 1 and we have determined $\sfZ^P_{d\k}(X)_{\ver}$. There is no $t$-dependence, of course, because the virtual dimension is already 0. This proves the first half of Theorem \ref{main}.

\section{Final formula with descendents}

Finally we compute the insertion term in \eqref{atlast}. We recall the definition of the coefficients $\gamma_a$ \eqref{gamma},
\begin{align*}
\sum_{a=0}^{\infty} \gamma_a X^a\ &=\ \prod_{i=1}^{m} \Big[E(t) \sum_{j=0}^{d-1} e^{X - j t} \Big]_{2\alpha_i} \\
&=\ \prod_{i=1}^{m} \sum_{j=0}^{d-1} \sum_{k=0}^{\alpha_i} \frac{(-t)^k}{(k+1)!} \big[ e^{X - j t} \Big]_{2(\alpha_i-k)} \hspace{2cm}
\end{align*}
\begin{align*}
\hspace{3cm} &=\ \prod_{i=1}^{m} \sum_{j=0}^{d-1} \sum_{k=0}^{\alpha_i} \frac{(-t)^k}{(k+1)!} \frac{1}{(\alpha_i - k)!} (X-jt)^{\alpha_i-k} \\
&=\ -\prod_{i=1}^{m} \frac{t^{\alpha_i}}{(\alpha_i+1)!}\sum_{j=0}^{d-1} \sum_{k=0}^{\alpha_i} \binom{\alpha_i+1}{\alpha_i - k}(-1)^{k+1} (Xt^{-1}-j)^{\alpha_i-k} \\
&=\ -\prod_{i=1}^{m} \frac{t^{\alpha_i}}{(\alpha_i+1)!}\sum_{j=0}^{d-1}
\left[(Xt^{-1}-j-1)^{\alpha_i+1}-(Xt^{-1}-j)^{\alpha_i+1}\right]
\end{align*}
by the binomial theorem. All terms of the sum cancel except for $j=0,d-1$, leaving
$$
\sum_{a=0}^{\infty} \gamma_a X^a\ =\ t^{|\alpha|}\prod_{j=1}^{m}
\frac{(Xt^{-1})^{\alpha_j+1}-(Xt^{-1}-d)^{\alpha_j+1}}{(\alpha_j+1)!}\,.
$$
Substituting
$$
X\ =\ \frac{dtQ^d}{Q^d-1}\ =\ \frac{dtQ^{d/2}}{Q^{d/2}-Q^{-d/2}}
$$
from the second line of \eqref{atlast} gives
$$
\sum_{a=0}^{\infty} \gamma_a X^a\ =\ t^{|\alpha|}\prod_{j=1}^{m}
\frac{d^{\alpha_j+1}}{(\alpha_j+1)!}\ 
\frac{Q^{d(\alpha_j+1)/2}-Q^{-d(\alpha_j+1)/2}}{(Q^{d/2}-Q^{-d/2})^{\alpha_j+1}}\,.
$$
Substituting this into \eqref{atlast} gives the proof of the second half of Theorem \ref{main}.

\begin{remark} \label{generalinsertions}
Consider \eqref{ZPver} for any insertion of the form
\[
\prod_{j=1}^{m_1} \tau_{\alpha_j}(D_j) \prod_{j=1}^{m_2} \tau_{\beta_j}(1),
\]
where $D_1, \ldots, D_{m_1} \in H_2(S)$ and $1 \in H^0(S)$. Recall the projections $\pi_1 : C^{[n_0]} \times C \rightarrow C^{[n_0]}$ and $\pi_2 : \Pic^{n_0}(C) \times C \rightarrow \Pic^{n_0}(C)$ of Section \ref{step2}. Expanding the explicit expression for the descendent integrands in Proposition \ref{E(x)} reduces \eqref{ZPver} to a linear combinations of integrals of the form
\[
\int_{C^{[n_0]}} \frac{1}{e(N^{\vir})}\Big[\cZ_0|_{C^{[n_0]} \times\{c_0\}}\Big]^a \, \prod_k \pi_{1*}  \big( [\cZ_0]^{b_k} \big),
\]
for some $a,b_k \geq 0$. Using $\O(\cZ_0) \cong (\AJ \times \id)^* \cP(1)$ gives
\begin{multline*}
At^{n_0} \int_{C^{[n_0]}}\eqref{grand} \, \prod_k \pi_{1*} \big( (\AJ \times \id)^*(\id+n_0[c_0]) + \pi_{1}^{*} \omega \big)^b \\ =\ At^{n_0}  \int_{C^{[n_0]}} \eqref{grand} \, \prod_k \big( \omega^{b_k} + b_k n_0 \omega^{b_k-1} - b_k (b_k-1) \omega^{b_k-2} \theta \big),
\end{multline*}
where \eqref{grand} is the same as before, and $A$ is the constant defined in \eqref{EFA}.
These integrals can be performed using the Poincar\'e formula \eqref{Poin} as before.
In this generality we are unable to re-sum the resulting expression to a closed formula.
\end{remark}

\section{Links to Gromov-Witten theory of $X$} \label{GWX}

In this Section we apply our results for stable pairs to Gromov-Witten theory, via the descendent-MNOP conjecture of Pandharipande-Pixton \cite{PP1}. We let
$$
\overline{M}_{\!g,m}^{\,\bullet}(S,\beta)\ =\ \overline{M}_{\!g,m}^{\,\bullet}(X,\iota_*\beta)^T
$$
be the moduli space of $m$-pointed stable maps of genus $g$ curves to $S$ in class $\beta$. The superscript ${\,\!}\udot$ indicates that we allow disconnected curves, but only stable maps which contract no connected components. The moduli space coincides --- as a Deligne-Mumford stack with perfect obstruction theory \cite[Proposition 3.2]{KT1} --- with the $T$-fixed locus of the corresponding moduli space of maps to $X$. As such it inherits a virtual normal bundle $N^\vir$ described, for instance, in \cite[Proposition 3.2]{KT1}, and we can define descendent invariants of $X$ by residues:
\beq{GWXdef}
N\udot_{g,\beta}(X,\tau_{\alpha_1}(\sigma_1) \cdots \tau_{\alpha_m}(\sigma_m))\ :=\
\int_{[\overline{M}_{\!g,m}^{\,\bullet}(S,\beta)]^{\vir}} \frac{1}{e(N^\vir)} \prod_{j=1}^{m} \tau_{\alpha_j}(\sigma_j).
\eeq
Here the descendent classes are defined in the usual way by
$$
\tau_{\alpha_j}(\sigma_j)\ :=\ \psi_{j}^{\alpha_j} \ev_{j}^{*} \sigma_j,
$$
where the $j$th $\psi$-class $\psi_j$ is the first Chern class of the cotangent line to the curve at the $j$th marked point.
Their generating function is 
\begin{equation*} \label{genfun2}
\sfZ_{\beta}^{GW}(X,\tau_{\alpha_1}(\sigma_1) \cdots \tau_{\alpha_m}(\sigma_m))\ :=\ \sum_g
N\udot_{g,\beta}(X,\tau_{\alpha_1}(\sigma_1) \cdots \tau_{\alpha_m}(\sigma_m))u^{2g-2},
\end{equation*}
where $g$ can be negative in disconnected Gromov-Witten theory.

When all descendence degrees are zero, the MNOP conjecture \cite{MNOP,PT1} states that $\sfZ_{\beta}^P(q)$ is a rational function invariant under $q\leftrightarrow q^{-1}$, and that substituting $q=-e^{iu}$ gives the Gromov-Witten generating function:
$$
\sfZ_{\beta}^{GW}(X,\tau_{0}(\sigma_1) \cdots \tau_{0}(\sigma_m))(u)\ =\ \sfZ_{\beta}^P(X,\tau_{0}
(\sigma_1) \cdots \tau_{0}(\sigma_m))(-e^{iu}).
$$
Therefore Theorem \ref{main0} gives the following obvious vanishing in Gromov-Witten theory. Since this can be proved more easily and directly by cosection localisation applied to $[\overline{M}_{\!g,m}^{\,\bullet}(S,\beta)]^{\vir}$ \cite{KL3}, it should perhaps be seen as a confirmation of the MNOP conjecture in this case.

\begin{proposition} \label{GWmain0'}
Suppose $S$ has a reduced irreducible canonical divisor. If the MNOP conjecture holds for $X = \mathrm{Tot}(K_S)$, then 
$$\sfZ_{\beta}^{GW}(X,\tau_{0}(\sigma_1) \cdots \tau_{0}(\sigma_m)) = 0,$$
unless $\beta$ is an integer multiple of the canonical class $\k$ and all $\sigma_j$ lie in $H^{\le2}(S)$.
\end{proposition}

Since the descendent-MNOP correspondence is linear, we may apply it to only the vertical contribution $Z^P_{\ver}$ to the stable pair generating function to give a ``vertical" contribution
$$
\sfZ_{d\k}^{GW}(X,\tau_{\alpha_1}(D_1) \cdots \tau_{\alpha_m}(D_m))_{\ver}
$$
to the Gromov-Witten generating function. We first study this for degree 0 insertions using the MNOP correspondence \eqref{MNOPP}.

\begin{proposition} \label{GWmain'}
Suppose $S$ has a smooth connected canonical divisor of genus $h=\k^2+1$, and the MNOP conjecture holds for $X = \mathrm{Tot}(K_S)$. Let $\mathsf P_d$ denote the product
$$
\prod_{j=1}^{\big\lfloor{\frac{d-1}2}\big\rfloor}\!2^{h-1}\Big[ d^2 + j^2 -jd +j(d-j) \cos(du) -d(d-j) \cos(ju) - jd \cos((d-j)u) \Big]^{h-1}\!.\!
$$
Then $\sfZ_{d\k}^{GW}(X)_{\ver}$ equals
$$
(-1)^{\chi(\O_S)}(-d)^{(h-1)(1-d)} \Big[ 2 \sin \Big( \frac{du}{2} \Big) \Big]^{2h-2} \Big[ d \cos \Big(\frac{du}{2} \Big) - d \Big]^{h-1}\mathsf P_d
$$
for $d$ even, and
$$
(-1)^{\chi(\O_S)}(-d)^{(h-1)(1-d)} \Big[ 2 \sin \Big( \frac{du}{2} \Big) \Big]^{2h-2}
\mathsf P_d
$$
for $d$ odd.  Furthermore
\beq{diveq}
\sfZ_{d\k}^{GW}(X,\tau_{0}(D_1) \cdots \tau_{0}(D_m))_{\ver}\ =\ \sfZ_{d\k}^{GW}(X)_{\ver}\prod_{j=1}^{m} (d\k \cdot\! D_j).
\eeq
For $d=1,2$ these are the complete 3-fold generating functions.
\end{proposition}


\begin{proof}
The generating function $\sfZ_{d \k}^{P}(X)_{\ver}$ of Theorem \ref{main} is invariant under $q \leftrightarrow q^{-1}$. More precisely, in the product $\prod_{j=1}^{d-1}(\cdots)$, mapping $q \leftrightarrow q^{-1}$ swaps the $j$th and $(d-j)$th terms. Setting $q = -e^{iu}$ and all $\alpha_j=0$ in Theorem \ref{main} gives the claimed formulae. Notice as a consistency check that the last formula \eqref{diveq} also follows from the divisor equation.
\end{proof}

The more general descendent-MNOP correspondence of \cite{PP1, PP2} also states that
$\sfZ_{d \k}^{P}(X,\tau_{\alpha_1}(\sigma_1) \cdots \tau_{\alpha_m}(\sigma_m))(q)$ is a rational function of $q$, and then (in this Calabi-Yau setting) that
\beq{MNOPP}
\sfZ_{d \k}^{P}\big(X,\tau_{\alpha_1}(\sigma_1) \cdots \tau_{\alpha_m}(\sigma_m)\big)(-e^{iu})\ =\ \sfZ_{d \k}^{GW}\big(X,\overline{\tau_{\alpha_1}(\sigma_1) \cdots \tau_{\alpha_m}(\sigma_m)}\,\big)(u)
\eeq
for any $\sigma_1, \ldots, \sigma_m \in H^{*}_{T}(X)$. 
Here the correspondence
\beq{bar}
\tau_{\alpha_1}(\sigma_1) \cdots \tau_{\alpha_m}(\sigma_m)\,\Mapsto\ \overline{\tau_{\alpha_1}(\sigma_1) \cdots \tau_{\alpha_m}(\sigma_m)}
\eeq
between descendents in the two theories is not the identity unless all $\alpha_j=0$. More generally it multiplies by a factor $(iu)^{-|\alpha|}$ and then adds corrections from stable maps where the evaluations of the marked points come together in $X$. These corrections are described by universal matrices\footnote{We will show that only for length-1 partitions $\mu,\nu$ do the matrices $\widetilde\sfK_{\mu\nu}$ contribute to our calculations. For these, $\widetilde\sfK_{\mu\nu}$ equals the simpler $\sfK_{\mu\nu}$ defined in \cite{PP1} by the ``capped descendent vertex".}
$$
\widetilde{\sfK}_{\mu\nu}\in \Q[i,c_1,c_2,c_3](\!(u)\!), \qquad i^2=-1,
$$
indexed by (finite, 2-dimensional) partitions $\mu, \nu$ and satisfying
\beq{24}
\widetilde{\sfK}_{\mu\nu}\ =\ 0\ \ \mathrm{unless}\ \ |\nu|+\ell(\nu)\,\le\,|\mu|+\ell(\mu)-3(\ell(\mu)-1),
\eeq
by \cite[Proposition 24]{PP1}.
(This makes the sum \eqref{PaPixgeneral} below finite.) For the $c_i$ we substitute the equivariant Chern classes of $T_X$. Assuming without loss of generality that $\alpha_1\ge\alpha_2\ge\cdots\ge\alpha_m$ and setting
\beq{shifted}
\mu:=(\alpha_1+1,\ldots,\alpha_m+1),
\eeq
the correspondence is
\begin{equation} \label{PaPixgeneral}
\overline{\tau_{\alpha_1}(\sigma_1) \cdots \tau_{\alpha_m}(\sigma_m)}\ :=\ \sum_P\pm  \prod_{S \in P} \sum_{\nu} \widetilde{\sfK}_{\mu\_S, \nu} \, \tau_\nu(\sigma_S).
\end{equation}
Here the first sum is over all \emph{set partitions} $P$ of $\{1, \ldots, m\}$ and the second sum is over all partitions $\nu$. The notation $\mu_S$ means the subpartition of $\mu$ defined by $S$, i.e. the partition whose elements are $\alpha_j+1$ for all $j$ in the subset $S$ of $\{1,\ldots,m\}$. Finally, for any permutation $\nu=(\nu_1,\ldots,\nu_\ell)$ of length $\ell=\ell(\nu)$,
\beq{tau}
\tau_\nu(\sigma_S)\ :=\ \psi_1^{\nu_1-1}\cdots\psi_\ell^{\nu_\ell-1}\cdot\ev_{1,\ldots,\ell}^*\Delta_*\bigg(\prod_{j \in S} \sigma_j\bigg),
\eeq
where $\Delta\colon X\to X^\ell$ is the small diagonal. 
For fixed $P$, the sign $\pm$ (which is always $+$ if all insertions $\sigma_i$ have even cohomological degree) in \eqref{PaPixgeneral} is dictated by the usual sign rules for differential forms: choose any ordering of the subsets $S_i\subset\{1,\ldots,m\}$, thus defining an order of the product $\prod_{S\in P}$ in \eqref{PaPixgeneral}. Each term of the product contains a further product $\prod_{j\in S_i}\sigma_j$ from \eqref{tau}. Multiplying them all together in this order gives a reordering of $\sigma_1\cdots\sigma_m$. Permuting it back to its original order (taking into account the degrees of the $\sigma_i$) produces the sign $\pm$.

The definition \eqref{tau} of $\tau_\nu(\sigma_S)$ may be rewritten in the following equivalent form. Let
$$
\sigma_S\ :=\ \prod_{j\in S}\sigma_j
$$
and write
\beq{kunn}
\Delta_* \sigma_S\ =\ \sum_j \theta_{j,1} \otimes \cdots \otimes \theta_{j,l}
\eeq
for its K\"unneth decomposition in $X^l$. Then
$$
\tau_{\nu}(\sigma_S)\ =\ \sum_j \tau_{\nu_1 - 1}(\theta_{j,1}) \cdots \tau_{\nu_\ell - 1}(\theta_{j,\ell}).
$$
The following will be useful to compute the K\"unneth decomposition \eqref{kunn}. We let $\Delta^S$ denote the small diagonal $S\to S^\ell$ and recall the projection $\pi\colon X\to S$ and zero section $\iota\colon S\into X$.

\begin{lemma} \label{Kunth} For $\sigma \in H^*(S)$,
$$
\Delta_{*\,} \pi^* \sigma\ =\ (\pi \times \cdots \times \pi)^* \, \Delta^{S}_{*} \, (\k^{\ell-1} \cdot \sigma).
$$
\end{lemma}

\begin{proof}
Since $(\iota\times\cdots\times\iota)\circ(\pi \times \cdots \times\pi)$ is homotopic to the identity, we have
$$
\Delta_{*} \pi^* \sigma\ =\ (\pi \times \cdots \times\pi)^*(\iota\times\cdots\times\iota)^*\Delta_{*} \pi^* \sigma.
$$
To compute the right hand side we write $\sigma=[A]$ as the Poincar\'e dual of homology class $A$. Then $\pi^*A$ is a Borel-Moore homology cycle of dimension 2 larger, and $(\iota\times\cdots\times\iota)^*\Delta_{*} \pi^* \sigma$ is the Poincar\'e dual of the intersection of $\Delta_*(\pi^*A)$ with $S\times\cdots\times S$.

First intersect with $S\times X^{\ell-1}$. This intersection is transverse and sends $\Delta_{*} \pi^*A$ to
$(\id\times\iota\times\cdots\times\iota)_*\Delta^S_*A$. Now intersect with $S^\ell\subset S\times X^{\ell-1}$. Since our cycle already sits inside this, the intersection simply caps with the Euler class of the normal bundle $\O_S\boxtimes K_S\boxtimes
\cdots\boxtimes K_S$ of this inclusion. Since this is $\k^{\ell-1}$ the result follows.
\end{proof}

We can now show that the descendent-MNOP correspondence applied to the stable pairs vanishing result Theorem \ref{main0} gives the analogous vanishing for Gromov-Witten invariants.

\begin{theorem} \label{GWmain0}
Suppose $S$ has a reduced irreducible canonical divisor. If the descendent-MNOP correspondence holds for $X = \mathrm{Tot}(K_S)$, then 
\beq{vanisheq}
\sfZ_{\beta}^{GW}(X,\tau_{\alpha_1}(\sigma_1) \cdots \tau_{\alpha_m}(\sigma_m)) = 0,
\eeq
unless $\beta$ is an integer multiple of the canonical class $\k$ and all $\sigma_j$ lie in $H^{\le2}(S)$.
\end{theorem}

\begin{proof}
For a fixed 3-fold $X$ (or for fixed values of $c_1,c_2,c_3$), and fixed curve class $\beta$, the descendent-MNOP correspondence is an invertible linear transformation \eqref{bar} on the free $\Q[i](\!(u)\!)$-module of descendent operators and their products. When ordered by total shifted descendence degree,\footnote{The shifting is due to the $\pm1$ shifting in \eqref{shifted} and \eqref{PaPixgeneral}. We define the shifted descendence degree of $\tau_\nu$ \eqref{PaPixgeneral} to be the size $|\nu|$ of the partition $\nu$. Thus, in these conventions, $\tau_\alpha=\tau_{(\alpha+1)}$ has shifted degree $\alpha+1$. The \emph{total} shifted degree of a product of descendents is then the sum of the individual shifted degrees.} it is an (infinite) lower triangular matrix with invertible diagonal entries. The diagonal terms come from the leading order term of \eqref{PaPixgeneral}, which is where $P$ is the finest partition $\{1\}\cup\cdots\cup\{m\}$ and $\nu=\mu_S$ in \eqref{PaPixgeneral}. Then each $S$ is a singleton $\{j\}$, $\mu_S=(\alpha_j+1)=\nu$ and \cite{PP1}
$$
\widetilde \sfK_{\mu\_S,\nu}=(iu)^{-\alpha_j}.
$$
All other terms of the same shifted descendence degree contribute zero by \eqref{24}. (So even though shifted descendence degree only defines a partial order, the lower triangular claim makes sense.)

Moreover, all corrections \eqref{PaPixgeneral} to the leading terms involve the same curve class $\beta$ and descendent insertions of classes in $H^*(S)$ which are products of the $\sigma_j$ and $\k$ (by Lemma \ref{Kunth}). Thus if $\beta\ne d\k$ or one of the $\sigma_j\in H^{\ge3}$ the same is true in the correction terms. For such classes, Theorem \ref{main0} gives vanishing of the stable pair invariants. Since the correspondence is invertible, we deduce the same vanishing for $Z^{GW}_\beta$ as for $Z^P_\beta$.
\end{proof}


\begin{lemma} \label{l=1}
Only partitions $\nu$ of length $\ell(\nu) = 1$ contribute to \eqref{MNOPP} via \eqref{PaPixgeneral}.
\end{lemma}

\begin{proof}
For general $\sigma_1, \ldots, \sigma_m \in H^{*}(S)$, fixed $S\subset\{1,\ldots,m\}$ and a partition $\nu$ with $\ell(\nu)>1$ we will show the contribution of $\tau_{\nu} (\sigma_S)$ to \eqref{MNOPP} --- via \eqref{PaPixgeneral} --- is zero. Let $d$ be the cohomological degree of $\sigma_S \in H^d(S)$.

If $\ell(\nu)=2$ then $\Delta_{*} \, \pi^* \, \sigma=(\pi \times \pi)^*\Delta^{S}_{*} \, (\k \cdot \sigma_S)$, with
$$
\Delta^{S}_{*} \, (\k \cdot \sigma_S) \ \in\ H^{d+6}(S \times S)\ \cong\ \bigoplus_{i+j=d+6}H^i(S)\otimes H^j(S).
$$
At least one of $i$ or $j$ is $\ge3$ in all of these summands, so their contribution vanishes by \eqref{vanisheq}.

If $\ell(\nu)=3$ then $\Delta_{*} \, \pi^* \, \sigma=(\pi\times\pi \times \pi)^*\Delta^{S}_{*} \, (\k^2 \cdot \sigma_S)$, 
where
$$
\Delta^{S}_{*} \, (\k^2 \cdot \sigma) \in H^{d+12}(S \times S \times S)\ =\ 
\bigoplus_{i+j+k=d+12}H^i(S)\otimes H^j(S)\otimes H^k(S).
$$
At least one of $i,j,k$ must always be $\ge4$, so again the contribution vanishes by \eqref{vanisheq}.
\end{proof}

\begin{proposition} \label{PaPix}
For any $\sigma_1, \ldots, \sigma_m \in H^{\geq 2}(S)$, the disconnected descendent generating function $\sfZ_{d \k}^{GW}(X,\tau_{\alpha_1}(\sigma_1) \cdots \tau_{\alpha_m}(\sigma_m))$ equals
$$
\sfZ_{d \k}^{P}\Bigg( \prod_{j=1}^{m} \sum_{b=1}^{\alpha_j+1} \widetilde{\sfK}_{(\alpha_j+1),(b)}^{-1} \Big|_{c_1 = t,\,c_2=c_3=0}\ \tau_{b-1}(\sigma_j) \Bigg),
$$
where $ \widetilde{\sfK}_{(\alpha_j+1),(b)}^{-1} \Big|_{c_1 = t,\,c_2=c_3=0}$ is the inverse of the infinite lower triangular matrix $$ \widetilde{\sfK}_{(a),(b)} \Big|_{c_1 = t,\,c_2=c_3=0} \in \Q[i,t](\!(u)\!).
$$
\end{proposition}
\begin{proof}
First we observe that the only set partition which contributes to \begin{equation} \label{GWbar}
\sfZ_{d\k}^{P}(X,\tau_{\alpha_1}(\sigma_1) \cdots \tau_{\alpha_m}(\sigma_m))\ =\ \sfZ_{d \k}^{GW}\big(X,\overline{\tau_{\alpha_1}(\sigma_1) \cdots \tau_{\alpha_m}(\sigma_m)}\,\big) \end{equation} is $P = \{1\} \cup \cdots \cup \{m\}$. Indeed for any other partition $P$, there is an $S \in P$ with $|S| \geq 2$ and $\sigma_S \in H^{\geq 4}$. Then for any partition $\nu$ of any length $\ell$, we have
$$
\Delta_* \sigma_S \in H^{\geq 6(\ell-1) + 4}(X^\ell).
$$
Since $6\ell-2>2\ell$ each summand of the K\"unneth decomposition of $\Delta_* \sigma_S$ must contain a class in $H^{\geq 3}$. This contributes zero to \eqref{GWbar} by Theorem \ref{GWmain0}.

Furthermore, by Lemma \ref{l=1}, only partitions $\nu$ of length $\ell(\nu)=1$ contribute via \eqref{PaPixgeneral} to \eqref{MNOPP}. Therefore (\ref{MNOPP}, \ref{PaPixgeneral}) simplify to
$$
\sfZ_{d\k}^{P}(X,\tau_{\alpha_1}(\sigma_1) \cdots \tau_{\alpha_m}(\sigma_m))\ =\  \sfZ_{d \k}^{GW}\Bigg(\prod_{j=1}^{m} \sum_{b=1}^{\alpha_j+1}  \tau_{b-1}\big(\widetilde{\sfK}_{(\alpha_j+1),(b)} \, \sigma_j\big) \Bigg),
$$
where the sign $\pm$ in \eqref{PaPixgeneral} is $+$ for the set partition $P = \{1\} \cup \cdots \cup \{m\}$.

The correspondence requires us to set $c_i$ to the $T$-equivariant $i$th Chern class of $X$. Using $T_X|_S = T_S\ \oplus\ K_S \otimes \mathfrak{t}$, we see that $c_1 = t$, $c_2 = c_2(S) - \k^2 - \k \, t$, and $c_3 = c_2(S) t$. Any occurrence of $c_1, c_2, c_3$ is multiplied by a class $\sigma_j \in H^{2}$ in \eqref{PaPixgeneral}. Therefore the terms $c_2(S) - \k^2 - \k \, t$ and $c_2(S) t$ contribute zero by Theorem \ref{GWmain0}. We get
$$
\sfZ_{d\k}^{P}(X,\tau_{\alpha_1}(\sigma_1) \cdots \tau_{\alpha_m}(\sigma_m))\ =\  \sfZ_{d \k}^{GW}\Bigg(\!\!\pm \prod_{j=1}^{m} \sum_{b=1}^{\alpha_j+1}  \widetilde{\sfK}_{(\alpha_j+1),(b)} \Big|_{c_1=t,\,c_2=c_3=0}\,\tau_{b-1}(\sigma_j) \!\Bigg).
$$
We suppress the specialisation $c_1=t,\,c_2=c_3=0$ from now on for notational brevity. Multiplying out,
\begin{align*}
\sfZ_{d\k}^{P}(X,\tau_{\alpha_1}(\sigma_1) \cdots \tau_{\alpha_m}(\sigma_m)) &=& \nonumber \\
\pm\sum_{b_1, \ldots, b_m} &&\hspace{-9mm} \prod_{j=1}^{m} \widetilde{\sfK}_{(\alpha_j+1),(b_j)}\, \sfZ_{d\k}^{GW}(X,\tau_{b_1-1}(\sigma_1) \cdots \tau_{b_m-1}(\sigma_m)),\label{multout}
\end{align*}
for any $\alpha_1, \ldots, \alpha_m$. Inverting gives
\begin{eqnarray*} 
\sfZ_{d\k}^{GW}(X,\tau_{\alpha_1}(\sigma_1) \cdots \tau_{\alpha_m}(\sigma_m)) &=& \\
\pm\sum_{b_1, \ldots, b_m} &&\hspace{-9mm} \prod_{j=1}^{m} \widetilde{\sfK}^{-1}_{(\alpha_j+1),(b_j)}\, \sfZ_{d\k}^{P}(X,\tau_{b_1-1}(\sigma_1) \cdots \tau_{b_m-1}(\sigma_m)).
\end{eqnarray*}
Expanding out the result we are required to prove gives precisely this.
\end{proof}

\begin{theorem} \label{GWmain}
Suppose $S$ has a smooth irreducible canonical divisor of genus $h=\k^2+1$ and the descendent-MNOP correspondence holds for $X = \mathrm{Tot}(K_S)$. Then $\sfZ_{d\k}^{GW}(X,\tau_{\alpha_1}(D_1) \cdots \tau_{\alpha_m}(D_m))_{\ver}$ equals the product of
$$
\sfZ_{d\k}^{GW}(X)_{\ver}\prod_{j=1}^m(d\k\cdot D_j)
$$ 
and
$$
\prod_{j=1}^{m} \sum_{b=1}^{\alpha_j+1} \widetilde\sfK_{(\alpha_j+1),(b)}^{-1} \Big|_{c_1 = t,\,c_2=c_3=0} \cdot  \frac{t^{b-1}}{b!} \bigg( \frac{- i d}{2} \bigg)^{\!b-1\,} \frac{\sin(b \, du/2 )}{\sin^{b}(du/2)}\,. 
$$
For $d=1,2$ these are the complete 3-fold generating functions.
\end{theorem}


\begin{proof}
Without descendents this is Proposition \ref{GWmain'}. The descendent term of Theorem \ref{main} is invariant under $q \leftrightarrow q^{-1}$ up to a sign $(-1)^{|\alpha|}$. Setting $-q = e^{iu}$ this term becomes
$$
t^{|\alpha|} \prod_{j=1}^{m} (d \k \cdot \! D_j) \left(\frac{-i}2\right)^{\!\alpha_j} \frac{d^{\alpha_j}}{(\alpha_j + 1)!} \frac{\sin((\alpha_j + 1) \, d u /2)}{\sin^{\alpha_j+1}(du/2)}\,.
$$
Combining with Proposition \ref{PaPix} and setting the sign $\pm$ to $+$ (because all $\sigma_j$ have even degree) gives the desired result.
\end{proof}

\section{Links to Gromov-Witten theory of $S$} \label{GWS}

The (disconnected) Gromov-Witten invariants of $S$,
\beq{GWSdef}
N\udot_{g,\beta}(S,\tau_{\alpha_1}(\sigma_1) \cdots \tau_{\alpha_m}(\sigma_m))\ :=\
\int_{[\overline{M}_{\!g,m}^{\,\bullet}(S,\beta)]^{\vir}} \prod_{j=1}^{m} \tau_{\alpha_j}(\sigma_j),
\eeq
can be recovered from those of $X$ \eqref{GWXdef} by taking the leading order term in their generating series.

\begin{lemma} \label{leading}
Define $g$ is so that the virtual dimension $g-1+\int_{\beta} c_1(S)+m$ of $\overline{M}_{\!g,m}^{\,\bullet}(S,\beta)$ equals the complex degree\footnote{The complex degree is half the cohomological degree.} of the descendent insertions:
\begin{equation} \label{dimvirg}
g - 1 + \int_{\beta} c_1(S) + m = \sum_{j=1}^{m} \Big(\alpha_j + \frac12\deg(\sigma_j)\Big). \vspace{-2mm}
\end{equation}
Then
$$
\sfZ_{\beta}^{GW}\!(X,\tau_{\alpha_1}(\sigma_1) \cdots \tau_{\alpha_m}(\sigma_m))\ =\ t^rN\udot_{g,\beta}(S,\tau_{\alpha_1}(\sigma_1) \cdots \tau_{\alpha_m}(\sigma_m))u^{2g-2} + O(u^{2g}),
$$
where $r=-\rk(N^{\vir})=g-1+\int_\beta c_1(S)=\sum_j\left(\alpha_j+\frac12\deg\sigma_j-1\right)$.
\end{lemma}

Of course this coefficient of $u^{2g-2}$ could be zero, in particular if $g$ defined by \eqref{dimvirg} is not an integer.

\begin{proof}
By \cite[Proposition 3.2]{KT1} the virtual normal bundle of $\overline{M}_{\!g,m}^{\,\bullet}(S,\beta)=\overline{M}_{\!g,m}^{\,\bullet}(X,\beta)^T\subset\overline{M}_{\!g,m}^{\,\bullet}(X,\beta)$ is
$$
N^\vir = R \pi_* f^* K_S \otimes \t,
$$
where
\begin{displaymath}
\xymatrix@R=18pt@C=5pt
{
\cC \ar^{f}[r] \ar^{\pi}[d] & S \\
\overline{M}_{\!g,m}^{\,\bullet}(S,\beta) &
}
\end{displaymath}
is the the universal curve. As in \cite[Section 3.1]{KT1}, by \eqref{param} this implies
$$
\frac{1}{e(N^{\vir})}\ =\ t^r+a_1t^{r-1}+a_2t^{r-2}+\cdots,
$$
with $a_i\in H^{2i}(\overline{M}_{\!g,m}^{\,\bullet}(S,\beta))$. Substituting into \eqref{GWXdef} gives
$$
N\udot_{g,\beta}(X,\tau_{\alpha_1}(\sigma_1) \cdots \tau_{\alpha_m}(\sigma_m))\ =\ 
t^rN\udot_{g,\beta}(S,\tau_{\alpha_1}(\sigma_1) \cdots \tau_{\alpha_m}(\sigma_m))
$$
for $g$ defined by \eqref{dimvirg}, while for smaller $g$ the left hand side vanishes for cohomological degree reasons.
\end{proof}
As a consequence, by Theorem \ref{GWmain0} we deduce the well known vanishing:
\begin{corollary} \label{GWSvanish}
Suppose $S$ has a smooth connected canonical divisor and let $g$ be defined by \eqref{dimvirg}. If the descendent-MNOP correspondence holds for $X = \mathrm{Tot}(K_S)$, then 
$$
N\udot_{g,\beta}(S,\tau_{\alpha_1}(\sigma_1)\cdots\tau_{\alpha_m}(\sigma_m))\ =\ 0,
$$
unless $\beta=d\k$ and all $\sigma_j$ lie in $H^{\le2}(S)$.$\hfill\square$
\end{corollary}
This was originally proved by Lee-Parker \cite{LP} using analytical techniques rather than the MNOP conjecture. See  \cite{MP} and \cite{KL1, KL2} for algebro-geometric proofs.  \medskip

So we consider $N\udot_{g,d\k}(S,\tau_{\alpha_1}(\sigma_1)\cdots\tau_{\alpha_m}(\sigma_m))$.
We first consider the case of no insertions. Thus $g=1+d\k^2$ \eqref{dimvirg} is the genus of degree $d$ \'etale covers $u\colon\Sigma\to C$ of genus $h=1+\k^2$. These covers are discrete, and Lee-Parker \cite{LP} show each contributes $(-1)^{h^0(u^* K_S|_C)}/|\Aut(u)|$ to the Gromov-Witten theory of $S$:
\begin{align} \label{Hurwitz}
N_{g,d\k}(S)\ =\ \sum_{u} \frac{(-1)^{h^0(u^* K_S|_C)}}{|\Aut(u)|}\,.
\end{align}
This was proved within algebraic geometry by Kiem-Li \cite{KL1, KL2}. The right hand side is the degree $d$ ``unramified spin Hurwitz number" of $(C,K_S|_C)$ --- the count of \'etale covers of $C$, signed by the parity of the theta characteristic\footnote{By the adjunction formula, $K_S|_C$ is a square root of $K_C$.} $K_S|_C$ of $C$.

\begin{corollary} \label{Guncor}
Suppose the smooth connected curve $C$ of genus $h$ is the canonical divisor of a smooth projective surface $S$, and that the MNOP conjecture holds for $X = \mathrm{Tot}(K_S)$. Then the vertical contribution to the unramified spin Hurwitz number \eqref{Hurwitz} is 
\beq{ttt}
(-1)^{\chi(\O_S)} \Bigg( \frac{2^{\frac{d-1}{2}}}{d!} \Bigg)^{\!2-2h}.
\eeq
For $d=1,2$ this is the entire unramified spin Hurwitz number \eqref{Hurwitz}. 
\end{corollary}

\begin{proof}
By Lemma \ref{leading} we must extract the coefficient of the leading term $u^{2g-2}$ of the 3-fold generating function of Proposition \ref{GWmain'}. In the product $\mathsf P_d$ all terms of order $\le3$ cancel, so we expand to order 4 via 
$$
\cos(x) = 1 - \frac{x^2}{2} + \frac{x^4}{24} + O(x^5).
$$
Expanding the remaining $\cos$ and $\sin$ terms to order $0$ and $1$ respectively easily gives the leading order term \eqref{ttt}.
\end{proof}

Using a TQFT formalism the spin Hurwitz numbers \eqref{Hurwitz} have been calculated by Gunningham \cite{Gun} as a sum over all strict partitions $\mu$ of $d$:
$$
\sum_{\mu\, \vdash d\textrm{\ strict}} (-1)^{\chi(\O_S) \, \ell(\mu)} \, ( d_\mu )^{2-2h}.
$$
Here $d_\mu$ is an explicit combinatorial number associated to $\mu$ and related to representations of the Sergeev algebra. The vertical contribution of Corollary \ref{Guncor} correctly reproduces the term corresponding to $\mu = (d)$.

In stable pairs theory partitions describe thickenings of the canonical divisor $C$, while in TQFT they parameterise irreducible representations of the symmetric group (the symmetry group of one fibre of an \'etale cover). Amazingly the MNOP correspondence seems to match these up. The calculations in the sequel \cite{KT4} provide further relations to Gunningham's formula. \medskip

Finally we consider \eqref{GWSdef} with divisorial descendents. For $d=1,2$, Maulik-Pandharipande \cite{MP} conjectured the following formulae
\begin{eqnarray*} \label{MPconj}
N\udot_{g,\k}(S,\tau_{\alpha_1}(D_1) \cdots \tau_{\alpha_m}(D_m))
&=& (-1)^{\chi(\O_S)} \prod_{j=1}^{m} (\k \cdot\! D_j) \frac{\alpha_j !}{(2\alpha_j +1)!} (-2)^{-\alpha_j}, \\
N\udot_{g,2\k}(S,\tau_{\alpha_1}(D_1) \cdots \tau_{\alpha_m}(D_m))
&=& (-1)^{\chi(\O_S)} 2^{h-1} \prod_{j=1}^{m} (2\k \cdot\! D_j) \frac{\alpha_j !}{(2\alpha_j +1)!} (-2)^{\alpha_j}.
\end{eqnarray*}
These formulae were proved by Kiem-Li \cite{KL1, KL2} using cosection localisation on the moduli space of stable maps, and later by Lee via symplectic geometry \cite{Lee}. We will show how their compatibility with our calculations shapes the form of the descendent-MNOP correspondence.

The leading term of the generating function $\sfZ_{d\k}^{GW}(X ,\tau_{\alpha_1}(D_1) \cdots \tau_{\alpha_{m}}(D_m))$ has order $u^{2g-2}$, where by \eqref{dimvirg},
$$
2g-2\ =\ d(2h-2) + 2|\alpha|.
$$
Similarly the leading order term of $Z^{GW}_{d\k}(X)$ has order $u^{d(2h-2)}$. Therefore
Theorem \ref{GWmain} implies the leading order term of
\begin{align} \label{degrees}
\prod_{j=1}^{m} \sum_{b=1}^{\alpha_j+1} \widetilde\sfK_{(\alpha_j+1),(b)}^{-1} \Big|_{c_1 = t,\,c_2=c_3=0} \cdot  \frac{t^{b-1}}{b!} \bigg( \frac{- i d}{2} \bigg)^{\!b-1\,} \frac{\sin(b \, du/2 )}{\sin^{b}(du/2)}\ =\ O(u^{2 |\alpha|})\!\!
\end{align}
is $u^{2g-2}u^{-d(2h-2)}=u^{2 |\alpha|}$.

We can substitute in the fact \cite[Theorems 2,\,3 and Section 7]{PP1} that the matrix $\widetilde\sfK_{(a),(b)}|_{c_1=t,\,c_2=c_3=0}$ vanishes unless $b\le a$, in which case
$$
\widetilde\sfK_{(a),(b)} \Big|_{c_1=t,\,c_2=c_3=0}\ =\ t^{a-b} f_{ab}(u),
$$
for some $f_{ab}(u) \in \Q[i](\!(u)\!)$ with $f_{aa} = (iu)^{1-a}$ . But since the $f_{ab}(u)$ could have many terms, \eqref{degrees} does not determine them.
 
\begin{conjecture} \label{conj1}
For each $a \geq b \geq 1$, we have $$\widetilde\sfK_{(a),(b)} \Big|_{c_1=t,\,c_2=c_3=0} \ =\ t^{a-b} K_{ab} \, u^{1-a},$$ for some \emph{constant} $K_{ab} \in \Q[i]$.
\end{conjecture}

If this is true, our results will shortly determine the $K_{ab}$; see \eqref{formulaL} below. There is a small amount of direct evidence for this conjecture. It is known to be true for $a=b$ (with $K_{aa} = i^{1-a}$ \cite[Theorem 2]{PP1}) and for $a=2,\,b=1$ (with $K_{21} = i^{-1}$ \cite[Section 2.5]{PP1}). Our motivation for it is the following.

\begin{theorem} \label{MPcor}
Fix $S$ with a smooth connected canonical divisor and $H_2(S)$ classes $D_1, \ldots,D_m$. Suppose the descendent-MNOP correspondence holds for $X = K_S$.
If Conjecture \ref{conj1} holds for $\widetilde\sfK$, then the vertical contribution to $N\udot_{g,d\k}(S,\tau_{\alpha_1}(D_1)\cdots\tau_{\alpha_m}(D_m))$ equals
$$
(-1)^{\chi(\O_S)} \Bigg( \frac{2^{\frac{d-1}{2}}}{d!} \Bigg)^{\!2-2h\,} \prod_{j=1}^{m} (d \k \cdot \! D_j) \frac{\alpha_j !}{(2\alpha_j +1)!} (-2)^{-\alpha_j} d^{2\alpha_j}.
$$
In particular, Maulik-Pandharipande's formulae for $d=1,2$ are true.
\end{theorem}

\begin{proof}
Note that Conjecture \ref{conj1} is equivalent to
\begin{equation*} \label{L}
\widetilde\sfK_{(a),(b)}^{-1} \Big|_{c_1=t,\,c_2=c_3=0}\ =\ t^{a-b} L_{ab} \, u^{b-1},
\end{equation*}
where $L_{ab} \in \Q[i]$ is the inverse of the infinite matrix $K_{ab}$. Setting $x:= \frac{d u}2$, the left hand side of \eqref{degrees} then becomes $t^{|\alpha|}$ times
$$
\prod_{j=1}^{m} \sum_{b=1}^{\alpha_j+1} \frac{(-i)^{b-1} L_{\alpha_j+1,b}}{b!} \, x^{b - 1} \frac{\sin b x}{\sin^{b} x}\,.
$$
Since
$$
\frac{(-i)^{\alpha_j} L_{\alpha_j+1,\alpha_j+1}}{(\alpha_j+1)!}\ =\ \frac{1}{(\alpha_j+1)!}\,,
$$
we can apply Theorem \ref{appthm} from the Appendix. By the uniqueness statement there, the coefficients $L_{ab}$ are \emph{uniquely} determined by the fact that \eqref{degrees} holds for $d=1$, $m=1$. From the second part of Theorem \ref{appthm} applied to $\alpha=\alpha_j+1$, we can then deduce that for any $d$ and $m$ we have
\begin{align*}
\prod_{j=1}^{m} \sum_{b=1}^{\alpha_j+1} \frac{(-i)^{b-1} L_{\alpha_{j+1},b}}{b!} \, x^{b - 1} \frac{\sin bx}{\sin^{b} x}\ =\ \prod_{j=1}^{m} \Bigg( \frac{(-1)^{\alpha_j}}{(2\alpha_j +1)!!} \, x^{2\alpha_j} + O(x^{2\alpha_j+1}) \Bigg).
\end{align*}
Substituting back $x=\frac{du}2$ gives
$$
\Bigg( \prod_{j=1}^{m} \frac{\alpha_j !}{(2\alpha_j+1)!}(-2)^{-\alpha_j}  d^{2\alpha_j}  \Bigg) u^{2|\alpha|} + O(x^{2|\alpha|+1}).
$$
Multiplying by the leading order term of $Z^{GW}_{d\k}(X)$ from \eqref{ttt} and taking the coefficient of $u^{2g-2}$ gives the required Gromov-Witten invariants of $S$.
\end{proof}

\begin{remark}
The proof of Theorem \ref{appthm} actually gives a formula \eqref{c} for the lower triangular matrix coefficients $\widetilde\sfK_{(a),(b)}^{-1}\big|_{c_1=t,\,c_2=c_3=0} = t^{a-b} L_{ab} u^{b-1}$, namely
\beq{formulaL}
L_{ab}\ =\ i^{b-1} (-1)^{a-1} \sum_{j=1}^{b} (-1)^{b-j} \binom{b}{j} j^{b-a} \quad\mathrm{for}\ a\ge b.
\eeq
By equation \eqref{Pix} of Appendix \ref{appA} we deduce a formula for the generating series of vertical contributions of all descendent Gromov-Witten invariants:
\begin{multline*}
\!\!\sum_{\alpha_1, \ldots, \alpha_m \geq 0} \sfZ_{d\k}^{GW}(X,\tau_{\alpha_1}(D_1) \cdots \tau_{\alpha_m}(D_m))_{\ver}  \,v_1^{\alpha_1}\cdots v_m^{\alpha_m}\ = \\
\sfZ_{d\k}^{GW}(X)_{\ver}\prod_{j=1}^m
\sum_{n \geq 1} \frac{\sin(n \, du/2 ) \, (du/2)^{n-1}}{\sin^{n}(du/2)} \frac{(d\k\cdot D_j)(t v_j)^{n}}{(t v_j)(t v_j+1) \cdots (t v_j+n)}\,,\!\!\!
\end{multline*}
where $v_1, \ldots, v_m$ are formal variables.
\end{remark}


\appendix
\addtocontents{toc}{\SkipTocEntry}
\section{\\ a\;generating function identity \\
\emph{by Aaron Pixton and Don Zagier}} \label{appA}
\addtocontents{toc}{\protect\contentsline{section}%
{\protect\tocsection{Appendix}{\thesection}%
{\hspace{15mm}A generating function identity, \emph{by A. Pixton and D. Zagier}}}%
{\thepage}}

\begin{theorem} \label{appthm}
For each $\alpha \in \Z_{>0}$, there exist unique $\{c_n(\alpha)\}_{n=1}^{\alpha}$ 
with $c_{\alpha}(\alpha) = \frac{(-1)^{\alpha-1}}{\alpha!}$ such that
\begin{equation} \label{order}
\sum_{n=1}^{\alpha} c_n(\alpha)\,\frac{x^n\sin(nx)}{\sin^nx}\ =\ A_\alpha x^{2\alpha-1}+ O(x^{2\alpha+1})
  \qquad\text{as $x\to0$}
\end{equation}
for some $A_\alpha\in\Q$. Moreover, the leading coefficient $A_\alpha$ is then
\begin{equation} \label{key}
A_\alpha \= \frac1{(2\alpha-1)!!} \= \frac1{(2\alpha-1)(2\alpha-3) \cdots1}\,.
\end{equation}
\end{theorem}
\begin{proof}
{\bf Existence.} We show that a solution to \eqref{order} is given by
\begin{equation} \label{c}
c_n(\alpha) \;:=\; \sum_{k=1}^{n} \frac{(-1)^{n-k} k^{n-\alpha}}{k! (n-k)!}\,.
\end{equation}
Notice that these $c_n(\alpha)$ equal the $n$th forward difference $\Delta^{n}\big(x^{n - \alpha} / n!\big)$
 at $x=0$ when $n > \alpha > 0$. Therefore they vanish in this range, and their generating series 
$L_\alpha(y) := \sum_{n=1}^{\infty} c_n(\alpha)\, y^n$ is a {\it polynomial} in $y$. We may 
therefore substitute $y= \frac{x e^{ix}}{\sin x}$ and split into real and imaginary parts, writing
$$
F_{\alpha}(x) \;:=\; L_{\alpha}\Bigl(\frac{x e^{ix}}{\sin x}\Bigr) \= E_{\alpha}(x) \+ i\,O_{\alpha}(x)
$$
where $E_{\alpha}(x) \in \Q[[x^2]]$ is even, while $O_{\alpha}(x) \in x \, \Q[[x^2]]$ 
is odd and equals the left hand side of \eqref{order}.

Splitting $n$ as $k+(n-k)$ in \eqref{c} gives $n c_{n}(\alpha) = c_n(\alpha-1) - c_{n-1}(\alpha-1)$, and so the 
recursive formula $L_\alpha'(y)=(y^{-1}-1)L_{\alpha-1}(y)$. Thus $F'_\alpha(x)=f(x)F_{\alpha-1}(x)$,
where 
$$ f(x)\= \left(\frac{\sin x}{xe^{ix}}\,-\,1\right)\frac d{dx}\!\left(\frac{xe^{ix}}{\sin x}\right) 
  \= x \+ \frac1x\,\left(1-\frac x{\tan x}\right)^2 \;\in\; x \+ x^3\,\Q[[x^2]]\,. $$
Taking even parts of this equation, we get
$$
O_{\alpha}^{\prime}(x) \,=\, f(x)\,O_{\alpha-1}(x)\,.
$$
Equation \eqref{order} with $A_\alpha$ as in equation~\eqref{key} now follows by induction on $\alpha$. \smallskip

\medskip  
\noindent {\bf Uniqueness.} The vector space $V_d=x\Q[[x^2]]/x^{2d+1}\Q[[x^2]]$ is $d$-dimensional
for every $d\ge1$, with basis $\{e_i=x^{2i-1}+x^{2d+1}\Q[[x^2]]\}_{1\le i\le d}$. Let $v_\alpha$ ($1\le\alpha\le d$)
be the image in $V_d$ of the left hand side of~\eqref{order} with coefficients given by~\eqref{c}. The
first part of the proof writes $v_\alpha$ as a linear combination of $e_\alpha,\dots,e_d$ with the
coefficient $A_\alpha$ of $e_\alpha$ being non-zero (and given by~\eqref{key}). It follows immediately
that these vectors are linearly independent and that no combination of the first~$\alpha$ of them
can be $\,\text O(x^{2\alpha+1})$, which is the desired uniqueness statement.
\end{proof}

\begin{remark}
For the application to Gromov-Witten theory the following formula for the generating series of the left hand 
sides of \eqref{order} is useful:
\begin{equation} \label{Pix}
\ \ \sum_{\alpha=1}^{\infty} (-v)^{\alpha-1} \sum_{n=1}^{\alpha} c_{n}(\alpha) \frac{x^n\sin(nx)}{\sin^nx}
  \ =\ \sum_{n=1}^{\infty}\frac{x^n\sin(nx)}{\sin^nx} \, \frac{v^n}{v(v+1) \cdots (v+n)}\,. 
\end{equation}
To prove it, we use partial fractions and geometric series expansions to get
$$  \frac1{(v+1)\cdots(v+n)} \= \sum_{k=1}^n\frac{(-1)^{k-1}}{(k-1)!(n-k)!}\,\frac1{v+k} 
  \= \sum_{\alpha=n}^\infty (-1)^{\alpha-1}\,c_n(\alpha)\,v^{\alpha-n}\,.$$
\end{remark}

\bigskip\noindent{\tt{m.kool1@uu.nl \\ richard.thomas@imperial.ac.uk}}
\end{document}